\documentclass[11pt]{amsart}
\usepackage{amsmath, amssymb, amsfonts,enumerate}
\usepackage[all]{xy}
\usepackage{amscd}
\usepackage{comment}
\usepackage{mathtools}
\usepackage{tikz} 
\usepackage{tikz-cd} 

 \newtheorem{theorem}{Theorem}[section]
\newtheorem{lemma}[theorem]{Lemma}
\newtheorem{corollary}[theorem]{Corollary}
\newtheorem{proposition}[theorem]{Proposition}
 
\newtheorem*{claim*}{Claim} 
 
 \theoremstyle{definition}
 \newtheorem{definition}[theorem]{Definition}
 \newtheorem{remark}[theorem]{Remark}

 \newtheorem{example}[theorem]{Example}

\newtheorem*{step*}{Step*}  

\numberwithin{equation}{section}
\newcommand {\Z}{\mathbb{Z}} 
\newcommand {\R}{\mathbb{R}} 
\newcommand {\Q}{\mathbb{Q}} 
\newcommand {\C}{\mathbb{C}} 

\newcommand{\AAA}{\mathcal{A}}

\newcommand{\HH}{\mathcal{H}}

\DeclareMathOperator{\Ker}{Ker}
\DeclareMathOperator{\End}{End}
\DeclareMathOperator{\Hom}{Hom}

\DeclareMathOperator{\Id}{Id}

\begin{document}
\title[Finiteness properties of algebraic group shifts]{On dynamical finiteness properties of algebraic group shifts}  
\author[X.K.Phung]{Xuan Kien Phung}
\email{phungxuankien1@gmail.com}
\begin{abstract}
Let $G$ be a  group  and let $V$ be an algebraic group over an algebraically closed field.  
We introduce algebraic group subshifts $\Sigma \subset V^G$ which generalize both the class 
of algebraic sofic subshifts of $V^G$ 
and the class of closed group subshifts over finite group alphabets. 
When $G$ is a polycyclic-by-finite group, we show that 
$V^G$ satisfies the descending chain condition and that 
the notion of algebraic group subshifts, the notion of 
algebraic group sofic subshifts, and that of algebraic group subshifts of finite type are all 
equivalent. 
Thus, we obtain extensions of well-known 
results of Kitchens and Schmidt 
to cover the case of many non-compact group alphabets.  

\end{abstract}
\date{\today}
\maketitle

\setcounter{tocdepth}{1}
 
\section{Introduction} 
\subsection*{Notation}
The set of integers and the set of real numbers are denoted 
by $\Z$ and $\R$ respectively.  
 Let $A, B$ be sets. Denote by $A^B$ the set of all maps from $B$ into $A$. 
Let $C \subset B$. If $x \in A^B$, the restriction $x\vert_C \in A^C$ is given by $x\vert_C(c) = x(c)$ for all $c \in C$.
If $X \subset A^B$, the restriction of $X$ to $C$ by defined as $X_C \coloneqq \{ x\vert_C \colon x \in X \} \subset A^C$. 
Fix an algebraically closed field $K$. 
An algebraic variety over $K$ is a reduced $K$-scheme of finite type and is identified with the set of $K$-points 
(cf.~\cite[Corollaire~6.4.2]{ega-1}). 
An algebraic group means a group that is also an algebraic variety with group operations given by algebraic morphisms (cf.~\cite{milne-group-book}). 
Algebraic subvarieties are Zariski closed subsets and algebraic subgroups are abstract subgroups which are also algebraic subvarieties.   
\par 
The main goal of this paper is to extend several  
well-known results 
on the descending chain condition and the finiteness property
of closed group subshifts whose alphabets are compact Lie groups to the context 
of algebraic group subshifts whose alphabets are algebraic groups over an algebraically closed field. 
By results of Tanaka and Chevalley (cf.~\cite{chevalley-lie-1}), 
the category of compact Lie groups is a certain subcategory of the category of $\R$-algebraic groups. 
Many algebraic groups are not complete and thus not compact when the base field is the field of complex  numbers  (e.g.~nontrivial affine spaces).  
 \par 
In order to state the results, we introduce some basic definitions of symbolic dynamics.  
Let  $A$ be a set, called the \emph{alphabet}, and let $G$ be a group, called the \emph{universe}.  
Elements of $A^G = \prod_{g \in G} A$ are called \emph{configurations} over the group $G$ and the alphabet $A$. 
The \emph{shift action} of the group $G$  on $A^G$ 
is defined by $(g,x) \mapsto g x$, where $gx(h) \coloneqq  x(g^{-1}h)$ for all  $x \in A^G$ and
$g,h \in G$.  
A configuration $x \in A^G$ is \emph{periodic} 
if its $G$-orbit is finite in $A^G$. 
If $A$ is a group then $A^G$ admits a natural group structure induced by pointwise group operation on each factor. 
\par 
The \emph{prodiscrete topology} on $A^G$ is the product topology where each factor $A$ is equipped with 
the discrete topology. With respect to this topology, $A^G$ is compact if and only if $A$ is finite. 
 To avoid confusion, when $A$ is a topological space and not merely a set, 
the \emph{Tychonoff topology} on $A^G$ will mean the product topology induced by the original topology on each factor $A$. 
By Tychonoff's theorem, $A^G$ is compact with respect to the Tychonoff topology if $A$ is compact. 
\par 
A  group $G$ is \emph{polycyclic-by-finite} if there exist subgroups $G=G_n \supset G_{n-1} \supset \cdots \supset G_0=\{1\}$ 
such that for every $1 \leq k \leq n$, $G_{k-1}$ is a normal subgroup of $G_{k}$ and $G_k/G_{k-1}$ is cyclic (cf.~\cite{passman-book-ring}). 
\par   
In \cite{kitchens-schmidt}, \cite[Theorem~4.2]{schmidt-book}, Kitchens and Schmidt  established 
the following important descending chain property (see, e.g, \cite{schmidt-book} for various applications):  
\begin{theorem}
[cf.~\cite{schmidt-book}]
\label{t:kitchens-schmidt-chain} 
Let $G$ be a  polycyclic-by-finite group and let $A$ be 
a compact Lie group. 
Let $(\Sigma_n)_{n \geq 0}$ be a descending sequence 
of $G$-invariant subgroups of $A^G$. 
Suppose that for all $n \geq 0$, $\Sigma_n$ is closed in $A^G$ 
with respect to the Tychonoff topology.  
Then the sequence  $(\Sigma_n)_{n \geq 0}$  eventually stabilizes. 
\end{theorem} 

The main object studied in this paper is the class of 
\emph{algebraic group subshifts} defined as follows. 
\begin{definition}
\label{d:alg-subgroup-shift} 
Let $G$ be a group and let $V$ be  an algebraic group over an algebraically closed field.  
A $G$-invariant subset $\Sigma \subset V^G$ is called an \emph{algebraic group subshift} of $V^G$ if it is 
closed in $V^G$ with respect to the prodiscrete topology and  
if the restriction $\Sigma_E \subset V^E$ is an algebraic subgroup 
for any finite subset $E \subset G$.        
\end{definition}  
With the above notations, an algebraic group subshift of $V^G$ is automatically an abstract subgroup (cf.~Proposition~\ref{p:alg-grp-sshift-subgroup}). 
\par 
As in the classical theory of symbolic dynamics, 
the closedness property in the full shift $A^G$, where $A$ is a set and $G$ is a group, 
 with respect to the prodiscrete topology of $A^G$ is the weakest closedness condition we must 
require for $G$-invariant subsets of $A^G$ to avoid pathologies. 
\par 
Our first main result (cf.~Theorem~\ref{t:main-polycyclic-finite} and Proposition~\ref{c:markov-alg}) 
is the following extension of Theorem \ref{t:kitchens-schmidt-chain}  
to algebraic group subshifts. 
\begin{theorem}
\label{t:intro-descending-property-alg-grp}
Let $G$ be a polycyclic-by-finite group. 
Let $V$ be an algebraic group over an algebraically closed field.  
Let $(\Sigma_n)_{n \geq 0}$ be a descending sequence of algebraic group subshifts of $V^G$. 
Then the sequence $(\Sigma_n)_{n \geq 0}$ eventually stabilizes.  
\end{theorem}
\par 
Now let $G$ be a group and let $A$ be a set.  
A  $G$-invariant subset  $\Sigma \subset A^G$ is called a \emph{subshift} of $A^G$. 
Remark that no condition on the closedness of such subshift $\Sigma$ in $A^G$ 
is required.
Given subsets $D \subset G$ and $P \subset A^D$, 
we define the following subshift of $A^G$: 
\begin{equation}
\label{e:sft} 
\Sigma(A^G; D,P) \coloneqq \{ x \in A^G \colon (g^{-1}  x)\vert_{D} \in P \text{ for all } g \in G\}. 
\end{equation} 
Such a set $D$ is called a
\emph{defining window} of $\Sigma(A^G; D,P)$. 
If $D$ is finite, $ \Sigma(A^G; D,P)$ is clearly closed in $A^G$ with respect to the prodiscrete topology 
 and it is called the 
\emph{subshift of finite type} of $A^G$ associated with $D$ and $P$. 
\par
Subshifts of finite type can be regarded as generalizations of higher dimensional topological Markov shifts. 
They are fundamental objects in various areas such as information theory and smooth dynamical systems  
(see e.g.~\cite{boyle-buzzi-gomez-2006}, \cite[Ch.~7]{kitchens-book}, \cite{sarig-1999}, \cite{sarig-2013}, \cite{lima-sarig-2019}, \cite{good-sft} and the references therein). 
\par 
In \cite{kitchens-schmidt}, \cite[Theorem~4.2]{schmidt-book}, the following remarkable finiteness result on closed group subshifts 
with compact Lie group alphabets is established: 
\begin{theorem}
[cf.~\cite{schmidt-book}]
\label{t:kitchens-schmidt-sft}
Let $G$ be a polycyclic-by-finite group and let $A$ be 
a compact Lie group. 
Let $\Sigma \subset A^G$ be a $G$-invariant subgroup   
which is closed in $A^G$ with respect to the Tychonoff topology. 
Then $\Sigma$ is a subshift of finite type of $A^G$.  
\end{theorem}

It turns out that an analogous result also holds for algebraic group subshifts. Before giving the statement, 
we recall a strong finiteness condition on algebraic group subshifts introduced in \cite{cscp-2020}.  

 \begin{definition}
[cf.~\cite{cscp-2020}] 
\label{d:algr-sft-intro} 
Let $G$ be a group and let $V$ be an  algebraic group over an algebraically closed field.   
A subset $\Sigma \subset V^G$ is an \emph{algebraic group subshift of finite type} 
if there is a finite subset $D \subset G$ and an algebraic subgroup $W \subset V^D$ such that 
 $\Sigma = \Sigma(A^G; D,W)$. 
\end{definition}
\par 
With the above notations, it is obvious from Definition \ref{d:algr-sft-intro} and the definition in \eqref{e:sft}  
that algebraic group subshifts of finite type are indeed subshifts of finite type so that they are closed 
with respect to the prodiscrete topology. 
\par
We establish the following result (cf.~Theorem~\ref{t:main-polycyclic-finite}) which extends Theorem~\ref{t:kitchens-schmidt-sft} 
to notably cover the case of  non-compact algebraic group alphabets. 
\begin{theorem}
\label{t:intro-sft-property-alg-grp}
Let $G$ be a polycyclic-by-finite group. 
Let $V$ be an algebraic group over an algebraically closed field. 
Let 
$\Sigma$ be an algebraic group subshift of $V^G$.   
Then  $\Sigma$ is an algebraic group subshift of finite type of $V^G$. 
\end{theorem}
  
Given sets $A, B$ and a group $G$, following the pioneering work of John von Neumann \cite{neumann-book}, 
a map $\tau \colon B^G  \to A^G$ is called a \emph{cellular automaton}  
 if there exist a finite subset $M \subset G$ called \emph{memory set} 
and a map $\mu \colon  B^M \to A$ called \emph{local defining map} such that 
\begin{equation} 
\label{e;local-property}
\tau(x)(g) = \mu( (g^{-1}x)\vert_M)  \quad  \text{for all } x \in B^G \text{ and } g \in G.
\end{equation}
\par 
Clearly, a cellular automaton $\tau \colon B^G \to A^G$ is uniformly continuous and $G$-equivariant 
(cf.~\cite{book}).  The converse also holds by the Curtis-Hedlund theorem \cite{hedlund}  
 when $A, B$ are finite and in the general case by the result of \cite{csc-curtis-hedlund}. 
\par 
Now let  $U, V$ be algebraic groups over an algebraically closed field $K$.  
Recall the following definition introduced in \cite{cscp-2020} (see also \cite{ccp-2019}, \cite{phung-2018} and \cite{gromov-esav}). 
A cellular automaton $\tau \colon U^G \to V^G$ is an \emph{algebraic group cellular automaton}  
if $\tau$ admits a memory set $M$ whose associated local defining map 
$\mu \colon U^M \to V$ is a homomorphism of $K$-algebraic groups.  
Given subshifts $\Sigma_1 \subset U^G$ and $\Sigma_2 \subset V^G$, 
a map $\tau \colon \Sigma_1 \to \Sigma_2$ is an \emph{algebraic group cellular automaton} 
if it is the restriction of some algebraic group cellular automaton $U^G \to V^G$. 
\par 
We denote by $\Hom_{U,V, G\text{-algr}} (\Sigma_1, \Sigma_2)$ the set 
of all algebraic group cellular automata $\Sigma_1 \to \Sigma_2$. 
When $U=V$ and $\Sigma_1= \Sigma_2= \Sigma$, 
we denote $\End_{V, G\text{-algr}}(\Sigma) \coloneqq \Hom_{V,V, G\text{-algr}} (\Sigma, \Sigma)$. 

\begin{definition}
[cf.~\cite{cscp-2020}]
\label{d:alg-sofic-shift} 
Let $G$ be a group and let $V$ be an algebraic group over an algebraically closed field.  
A subset $\Sigma \subset V^G$  is an \emph{algebraic  group sofic subshift} 
if it is the image of an algebraic group subshift of finite type $\Sigma' \subset U^G$, 
where $U$ is a $K$-algebraic group, 
under an algebraic group cellular automaton $U^G \to V^G$.  
\end{definition} 
\par 
When the universe $G$ is countable, algebraic group sofic subshifts are indeed algebraic group subshifts  
(cf.~Theorem~\ref{t:sofic-alg-grp}). 
Thus we obtain the following direct consequence of Theorem~\ref{t:intro-sft-property-alg-grp}: 
\begin{corollary} 
Let $G$ be a polycyclic-by-finite group.  
Let $V$ be an algebraic group over an algebraically closed field. Let  
$\Sigma$ be an algebraic group sofic subshift of $V^G$. 
Then $\Sigma$ is an algebraic group subshift of finite type of $V^G$.
\qed 
\end{corollary}

\par 
Moreover, we obtain in this paper some first  finiteness results in the category of algebraic group subshifts 
(cf.~Section~\ref{s:applications}).  
\begin{theorem} 
\label{c:intro-algebraic-group-abelian} 
Let $G$ be a polycyclic-by-finite group. 
Let $U, V$ be algebraic groups over the same algebraically closed field. 
Let $\Sigma_1$ and  $\Sigma_2$ be respectively algebraic group subshifts 
of $U^G$ and  $V^G$.  
Let $\tau \in \Hom_{U,V, G\text{-algr}} (U^G, V^G)$. 
Then $\tau^{-1}(\Sigma_2)$ and $\tau(\Sigma_1)$  
are algebraic group subshifts of finite type of $U^G$ and $V^G$ respectively.  
\end{theorem} 
 \par 
Given a map $f \colon X \to X$ from a set $X$ to itself, the \emph{limit set} $\Omega(f)$ of $f$ is 
defined as the intersection of the images of its iterates, i.e., 
$\Omega(f) \coloneqq \bigcap_{n \geq 1} f^n(X)$, where 
$f ^n \coloneqq f \circ \cdots \circ f$ ($n$-times).  
The map $f$ is said to be \emph{stable} if 
$f^{n+1} (X) = f^n(X)$ for some integer $n \geq 1$. 
\par 
The notion of limit sets is introduced by 
Wolfram~\cite{wolfram-univ-1984}  and were subsequently studied notably in 
\cite{milnor-ca-1988}, \cite{culik-limit-sets-1989} \cite{kari-nilpotency-1992}, \cite{guillon-richard-2008}, \cite{cscp-2020}.  
\par 
We give another direct application of our main results 
on the dynamics and the limit sets of endomorphisms of algebraic group subshifts. 
\begin{corollary}
\label{c:limit-set-alg-grp} 
Let $G$ be a polycyclic-by-finite group and let $V$ be an algebraic group over an algebraically closed field. 
Let $\Sigma$ be an algebraic group subshift of $V^G$ and 
let $\tau \in \End_{V, G\text{-algr}} (\Sigma)$. 
Then the limit set $\Omega(\tau)$ is an algebraic group subshift of finite type of $V^G$ and $\tau$ is stable. 
\end{corollary}
  
\begin{proof}
Theorem~\ref{t:sofic-alg-grp} shows that the iterated images $\tau^n(\Sigma)$ are algebraic group subshifts of $V^G$. 
Then Theorem~\ref{c:descending-property} on the intersection of algebraic group subshifts  implies that 
$ \Omega(\tau) \coloneqq \bigcap_{n \geq 1} \tau^n(\Sigma)$ is also an algebraic group subshift of $V^G$.  
Consequently, $\Omega(\tau) $ is 
an algebraic group subshift of finite type of $V^G$ by Theorem~\ref{t:intro-sft-property-alg-grp}. 
Hence, it follows from \cite[Theorem~1.3.(iv)]{cscp-2020} that $\tau$ must be stable. 
\end{proof} 

It turns out that the proofs of our main theorems can be applied verbatim to obtain 
generalizations (cf.~Section~\ref{s:generalization}) to the so called \emph{admissible group subshifts} 
(cf.~Definition~\ref{d:admissible-subgroup-shift}). 
For example, 
we obtain the following finiteness result (cf.~Theorem~\ref{t:artinian-general} and 
Example~\ref{example:canonical-admissible-for-intro}):   
\begin{theorem}
\label{t:artinian-intro} 
Let $G$ be a polycyclic-by-finite group. Let $A$ be an Artinian group (resp. an Artinian module).  
Then every closed subshift of $A^G$ which is also an abstract subgroup (resp. a submodule) 
is a subshift of finite type. 
\end{theorem}

Here, a group $\Gamma$ is \emph{Artinian} if every descending sequence of subgroups of $\Gamma$ eventually stabilizes. 
\par 
It is not known whether 
any of the above finiteness results still holds for some universe $G$ which is not a polycyclic-by-finite group. 
Note that if a group $G$ admits a non-finitely-generated subgroup, then 
there exist a finite group $A$ and a closed subshift $\Sigma \subset A^G$ 
which is also an abstract subgroup but $\Sigma$ is not a subshift of finite type  
(cf.~\cite{osin-02} and \cite{salo-2018-polycyclic}).  
\par 
The paper is organized as follows. 
Section~\ref{s:prelim} provides some   
lemmata on the window change of subshifts and recalls the machinery of inverse systems.  
In Section~\ref{s:rest-close}, we show that 
the closedness in the prodiscrete topology of a subset in the full shift 
is stable under restriction to arbitrary subsets under mild algebraic hypotheses (Lemma~\ref{l:restriction-closed},  
Lemma~\ref{l:restriction-closed-complete}). 
We then define in Section~\ref{s:alg-algr-shift} the class of \emph{algebraic subshifts} which generalizes   
 algebraic sofic subshifts introduced in \cite{cscp-2020}. 
 Proposition~\ref{p:alg-grp-sshift-subgroup} shows that algebraic group subshifts are automatically abstract subgroups 
 and all algebraic group sofic subshifts are algebraic group subshifts. 
Theorem~\ref{c:descending-property}   shows that arbitrary intersections of algebraic group subshifts are also algebraic group subshifts. \par 
We obtain in Section~\ref{s:markov-type-alg} the equivalence between the descending chain condition and the finite type property of algebraic group subshifts 
 (Proposition~\ref{c:markov-alg}). This leads us to extend the notion of \emph{groups of Markov type} in   \cite[Definition~4.1]{schmidt-book} to  introduce the class of \emph{groups of algebraic Markov type}  (Definition~\ref{d:markov-type-alg}) 
 which is stable by taking subgroups (Proposition~\ref{p:markov-subgroup}) 
 and by extensions by cyclic groups (Corollary~\ref{c:markov-group-extension}). 
 \par 
 Section~\ref{s:main} contains the main technical result Theorem~\ref{t:algebraic-group-Z-d} whose proof requires new ideas and techniques to overcome the absence of the compactness hypothesis on the alphabets. 
We prove in  Theorem~\ref{t:main-polycyclic-finite} that every polycyclic-by-finite group is of algebraic Markov type  
which implies 
both Theorem~\ref{t:intro-descending-property-alg-grp} and Theorem~\ref{t:intro-sft-property-alg-grp}. 
Section~\ref{s:applications} studies inverse images of homomorphisms of algebraic group subshifts and proves
 Theorem~\ref{c:intro-algebraic-group-abelian}. 
We then introduce \emph{admissible Artinian group structures} (Definition~\ref{d:general-artinian-structure}) that are groups 
equipped with a certain collection of subgroups satisfying the descending chain condition.  
A common generalization of the finiteness results in the 
Introduction is then obtained in Theorem~\ref{t:artinian-general} 
for the class of \emph{admissible group subshifts} whose alphabets are admissible Artinian group structures. \par 
Finally, Section~\ref{s:density} gives some applications on the density of periodic configurations for certain algebraic subshifts 
(Theorem~\ref{density-complete-z}) and for admissible group subshifts (Corollary~\ref{density-alg-grp-z}, Corollary~\ref{density-alg-grp-res-finite}). Another application of the techniques developed in this paper can be found in \cite{phung-2020-shadow} where it is shown that an intrinsic shadowing property is always satisfied for the valuation action of group cellular automata on an admissible group subshift. 
  
 \section{Preliminaries}
\label{s:prelim}

To fix the notations, suppose that $G$ is a group and let $E,F$ be subsets of $G$. 
Then we write $E F \coloneqq \{g h : g \in E, h \in F\}$. 
 
\subsection{Window change lemmata for subshifts}

For subshifts described by the formula \eqref{e:sft}, 
we have the following easy but useful observation. 

\begin{lemma}
\label{l:base-change-sft}
Let $G$ be a group and let $A$ be a set. 
 Let $D \subset G$ be a subset and let $P \subset A^D$. 
 Let $\Sigma= \Sigma(A^G; D,P) \subset A^G$.  
Then, for every subset $E \subset G$ such that $D \subset E$, we have 
$\Sigma= \Sigma(A^G; E, \Sigma_E)$. 
\end{lemma} 

\begin{proof}
The proof is similar to \cite[Lemma~5.1]{cscp-2020}.  
\end{proof}

Now suppose that $H$ is a subgroup of a group $G$. 
Let $E \subset G$ be a subset  
such that $Hk_1\neq Hk_2$ for all $k_1, k_2 \in E$ such that $k_1 \neq k_2$.  
Let $A$ be a set. 
Denote $B \coloneqq A^E$. 
Then for every subset $F \subset H$, 
we have a  canonical bijection $B^F = A^{F   E}$ defined as follows.  
For every $x \in B^F$, we associate an element $y \in A^{F   E}$ given  
by $y(h k) \coloneqq (x(h))(k)$ for every $h \in F$ and $k \in E$. 
The obtained bijection is clearly functorial with respect to inclusions $ F \subset F'$ of subsets of $H$ in a trivial way. 
\begin{lemma}
\label{l:restriction-sft} 
With the above notations and hypotheses, let $D \subset H$ and let $P \subset A^{DE}=B^D$ be subsets. 
Then the following equality between subshifts of $A^G$ holds: 
\[
\Sigma(A^G; HE, \Sigma(B^H; D, P)) = \Sigma(A^G; DE, P). 
\]
\end{lemma}

\begin{proof} 
Let us denote $\Sigma \coloneqq \Sigma(A^G; H E, \Sigma(B^H; D, P))$.  
\par 
Let $x \in \Sigma$  and let $g \in G$. 
Then  we have $(g^{-1}x) \vert_{H E} \in \Sigma(B^H; D, P)$. 
Since $D   E \subset H   E$, 
we deduce from the canonical bijection $A^{DE} = B^D$ and the definition of $\Sigma(B^H; D, P)$ 
that $  (g^{-1}x) \vert_{D   E} \in  P$. 
Thus, we find that $\Sigma \subset \Sigma(A^G; DE, P)$. 
\par 
Conversely, let 
$x \in \Sigma(A^G; DE, P)$ and let $g \in G$. 
Then it follows that 
$(g^{-1}x)\vert_{DE} \in P$. 
Since $D  E \subset H E$, 
we have  
$((g^{-1}x)\vert_{H E})\vert_{D E} \in P$.  
Therefore, $(g^{-1}x)\vert_{H E} \in \Sigma(B^H; D, P)$. 
Hence, we deduce that 
$x \in \Sigma$ so that $\Sigma(A^G; DE, P) \subset \Sigma$ and the conclusion follows. 
\end{proof}

\subsection{Restriction, induction and closedness} 
Let $A$ be a set. 
Let $H$ be a subgroup of a countable group $G$ and let $\Lambda \subset V^H$ be a subshift. 
Let  
$\Sigma \coloneqq \Sigma (V^G ; H, \Lambda)$ (cf.~\eqref{e:sft}). 
We call $\Sigma$ the \emph{induction subshift} associated with $\Lambda$ and the subgroup $H$ of $G$.  
\par 
The following lemma will be used to prove Lemma~\ref{l:alg-shift-induction} and Proposition~\ref{p:markov-subgroup}.   

\begin{lemma}
\label{l:restrict-closed-subshift}
Let the notations and hypotheses be as above. 
Then we have $\Sigma_H = \Lambda$. Moreover,  
with respect to the prodiscrete topology, 
$\Sigma$ is closed in $A^G$ if and only if $\Lambda$ is closed in $A^H$. 
\end{lemma}

\begin{proof}
We have a canonical factorization 
$\Sigma = \prod_{c \in G/H} \Sigma_c$ 
where each $x \in \Sigma$ is identified with $(x\vert_c)_{c \in G/H} \in \prod_{c \in G/H} \Sigma_c$ (cf.~\cite[Section~2.5]{cscp-2020}). 
Recall  that $\Sigma_c \coloneqq \{ z\vert_c \colon z \in \Sigma \}$ 
for every left coset $c \in G/H$. 
\par
For every $c \in G/H$, we choose $g_c \in c \subset G$ 
and thus obtain a homeomorphism 
$\phi_c \colon \Sigma_c \to \Sigma_H$ given by
$\phi_c(y)(h) \coloneqq y(g_ch)$ for all $y \in \Sigma_c$. 
If $c=H$, we choose $g_c=1_G$ so that $\phi_H = \Id_{\Sigma_H}$.  
\par 
 By definition of $\Sigma$, we have $\Sigma_H \subset \Lambda$. 
 Conversely, let $z \in \Lambda$ then it is the restriction to $H$ of the configuration 
  $x = (x_c)_{c \in G/H} \in \Sigma$ defined by $x_c= g_c z$ for $c \in G/H$.   
 Thus  $\Lambda = \Sigma_H$ and $\phi_c(\Sigma_c)= \Lambda$ for every $c \in G/H$. 
 \par
 As $G$ is countable, there exists an increasing sequence $(E_n)_{n\geq 0}$ of finite subsets of $G$ 
 such that $G= \bigcup_{n \geq 0} E_n$. 
Then for every $c \in G/H$, the sets $E_n \cap c$ are finite and form an increasing sequence 
 whose union is $c$. 
 \par 
Suppose first that $\Lambda$ is closed in $A^H$. 
Let $(y_n)_{n \geq 0}$ be a sequence in $\Sigma$ which converges to some $z \in A^G$. 
Then for   $c \in G/H$, the sequence $(y_n\vert_c)_{n \geq 0}$ in $\Sigma_c$ converges to 
$z\vert_c \in A^c$. Since $\Lambda$ is closed, so is $\Sigma_c = \phi^{-1}(\Lambda)$ 
and thus $z\vert_c \in \Sigma_c$. Hence, $z=(z\vert_c)_{c \in G/H} \in \prod_{c \in G/H} \Sigma_c= \Sigma$ 
so $\Sigma$ is closed in $A^G$. 
\par 
Conversely, suppose that $\Sigma$ is closed in $A^G$. Let $(z_n)_{n \geq 0}$ be a sequence in $\Lambda$ which 
converges to some $z \in A^H$. For every $n \geq 0$, 
define $y_n  \in \Sigma$ by setting  $(y_n)_c= g_c z_n \in A^c$ for   $c \in G/H$. 
Let $y\in A^G$ be given by $y_c=g_c z \in A^c$  for all $c \in G/H$. 
Then it is clear that $(y_n)_{n \geq 0}$ converges to $y$ in $A^G$. Since $\Sigma$ is closed, it follows that $y \in \Sigma$ 
and thus $z = y_H \in \Sigma_H= \Lambda$. We conclude that $\Lambda$ is closed in $A^H$.   
\end{proof}

\subsection{Inverse limits of closed algebraic inverse systems} 
 Let $I$ be a directed set, i.e., a partially ordered set in which every pair of elements has an upper bound.
An \emph{inverse system} of sets  \emph{indexed} by  $I$ consists of a set $X_i$ for each $i \in I$ and 
a \emph{transition map} $\varphi_{ij} \colon X_j \to X_i$ for all $i,j \in I$ such that $i \prec j$. 
We require that transition maps satisfy the following compatibility conditions: 
\begin{align*}
 \varphi_{ii} &= \Id_{X_i} \quad  \text{  for all } i \in  I, \\
 \varphi_{ij} \circ \varphi_{jk}  
 & = \varphi_{ik}  \quad  \,\,\, \text{ for all $i,j,k \in I$ with } i \prec j \prec k. 
\end{align*}
When the index set and the transition maps are clear, 
we simply say an inverse system $(X_i)_{i \in I}$. 
 \par
The \emph{inverse limit} of an inverse system $(X_i)_{i \in I}$ with transition maps $\varphi_{ij} \colon X_j \to X_i$ 
is defined as the following subset of $\subset \prod_{i \in I} X_i$: 
\[
\varprojlim_{i \in I} (X_i,\varphi_{i j}) =  \varprojlim_{i \in I} X_i \subset \prod_{i \in I} X_i
\]
which consists of all $(x_i)_{i \in I}$ such that $\varphi_{i j}(x_j)= x_i$ for all $i \prec j$. 
 \par 
In this paper, we shall use repeatedly the following lemma which gives a sufficient condition for the nonemptiness 
 of an inverse system. 
\begin{lemma}
\label{l:inverse-limit-alg-grp}
Let $K$ be an  algebraically closed field.
Let $(X_i, f_{ij})$ be an  inverse system indexed by a countable directed set $I$, 
where    each $X_i$ is a nonempty  $K$-algebraic variety and each transition map  
$f_{i j} \colon X_j \to X_i$ is an   algebraic morphism
such that $f_{ij}(X_j) \subset X_i$ is a Zariski closed subset for all $i \prec j$. 
Then $\varprojlim_{i \in I} X_i \neq \varnothing$.
\end{lemma} 

\begin{proof}
The statement is proved in \cite[Proposition 4.2]{phung-2018}. 
\end{proof}

\subsection{Inverse limits and closed subsets in the prodiscrete topology}  

Let $G$ be a countable set and let $(E_n)_{n \geq 0}$ 
be an increasing sequence of finite subsets of $G$ such that $\bigcup_{n \geq 0} E_n = G$. 
Let $A$ be a set and let $\Sigma \subset A^G$ be a subset. 
For every $m \geq n \geq 0$, 
the inclusion $E_n \subset E_m$ induces a canonical projection $A^{E_m} \to A^{E_n}$ which in turns induces 
a well-defined map $\pi_{nm} \colon \Sigma_{E_m} \to \Sigma_{E_n}$. 
Thus, we obtain an inverse system $(\Sigma_{E_n})_{n \geq 0}$ with transition maps $\pi_{nm}$ for $m \geq n \geq 0$. 
We have the following useful approximation result (cf.~\cite[Section~4]{cscp-2020}): 

\begin{lemma} 
\label{l:closed-lim}
With the above notations, suppose that $\Sigma$ is closed in $A^G$ with respect to the prodiscrete topology. 
Then we have a canonical bijection 
\[
\varprojlim_{n \geq 0} \Sigma_{E_n} = \Sigma. 
\]
\end{lemma}

\begin{proof}
For every $x \in \Sigma$, we have $x_n \coloneqq x\vert_{E_n} \in \Sigma_{E_n}$ 
and we can clearly associate an element $(x_n)_{n \geq 0} \in \varprojlim_{n \geq 0} \Sigma_{E_n}$. 
\par 
Conversely, let $x = (x_n)_{n\geq 0} \in \varprojlim_{n \geq 0} \Sigma_{E_n}$. 
Then $x_n \in \Sigma_{E_n}$ for every $n \geq 0$. As the transition maps 
$\pi_{nm}$ are simply projection maps and as 
$\bigcup_{n \geq 0} E_n = G$, 
we can associate a well-defined configuration $x \in A^G$ given by $x(g) \coloneqq x_n(g)$ for every 
$g \in G$ and every $n \geq 0$ large enough such that $g \in E_n$. 
Thus $x\vert_{E_n} = x_n \in \Sigma_{E_n}$ for every $n \geq 0$. 
Since every finite subset of $G$ is contained in some set $E_n$ for some $n \geq 0$,  
we deduce that $x$ belongs to the closure of $\Sigma$ in $A^G$ with respect to the prodiscrete topology.  
The closedness of $\Sigma$ then implies that $x \in \Sigma$. 
\par
It is obvious that the above operations are mutually inverse. The proof is completed. 
\end{proof}

\section{Restriction and closedness property}
\label{s:rest-close}
In this section, we  show that under suitable algebraic assumptions, 
the closedness in the prodiscrete topology is stable under restriction to arbitrary subsets. 
\par 
Suppose that $G$ is a countable set.  
Let $A$ be a set and let $\Sigma$ 
be a closed subset of $A^G$ with 
respect to the prodiscrete topology.    
Let $K$ be an algebraically closed field.

\begin{lemma}
\label{l:restriction-closed} 
With the above notations and hypotheses,  suppose in addition that $\Sigma$ 
satisfies the following property: 
\begin{enumerate} [\rm (P)] 
\item 
for every finite subset $E \subset G$,  
$\Sigma_E$ is a $K$-algebraic group and for all finite subset $F \subset E$, 
the induced projection 
$\pi_{F,E}  \colon \Sigma_E \to \Sigma_F$ is a homomorphism of $K$-algebraic groups. 
\end{enumerate}
Then for every subset $H \subset G$, 
the restriction $\Sigma_H$ is also closed in $A^G$ with respect to the prodiscrete topology. 
\end{lemma}

\begin{proof} 
Since $G$ is countable, there exists  an increasing sequence  $(E_n)_{n \geq 0}$  
of finite subsets of $G$ such that $G= \bigcup_{n \geq 0} E_n$. 
For every $n \geq 0$, denote $F_n \coloneqq H \cap E_n$ then clearly $H= \bigcup_{n \geq 0} F_n$. 
\par 
Let $d \in A^G$ belong to the closure of $\Sigma_H$ in $A^H$ 
with respect to the prodiscrete topology. 
Then, for every $n \geq 0$, we have $d\vert_{E_n} \in \Sigma_{E_n}$. 
Hence, we obtain for every $n \geq 0$ a nonempty  subset of $\Sigma_{E_n}$ as follows: 
\[
Z_n \coloneqq \{ x \in \Sigma_{E_n} \colon x\vert_{F_n} = d\vert_{F_n}\}.   
\] 
It is clear that for every $m \geq n \geq 0$, the restriction of $\pi_{E_n, E_m}$ to $Z_{m}$
 induces a well-defined map  $p_{nm} \colon Z_m \to Z_n$. 
Since $\pi_{E_n, F_n}  \colon \Sigma_{E_n} \to \Sigma_{F_n}$ is a homomorphism of $K$-algebraic groups by the hypothesis (P),  
it follows that $Z_n = \pi_{E_n, F_n}^{-1} (d\vert_{F_n})$ is a translate of an algebraic subgroup of $\Sigma_{E_n}$. 
Therefore, for every $m \geq n \geq 0$, the transition map $p_{nm}$ of the inverse system $(Z_n)_{n \geq 0}$ 
 is a morphism of algebraic varieties and 
$p_{nm} (Z_m)$ is Zariski closed in $Z_n$. 
\par 
Hence, Lemma~\ref{l:inverse-limit-alg-grp} applied to the inverse system $(Z_n)_{n \geq 0}$ tells us 
that there exists $x \in \varprojlim_{n \geq 0} Z_n \subset A^G$. 
We infer from the definition of the sets $Z_n$ that $x\vert_{F_n} = d\vert_{F_n}$ for every $n \geq 0$. 
This shows that $x\vert_H = d \in A^H$. 
As $\Sigma$ is closed in $A^G$ and $\bigcup_{n \geq 0} E_n = G$, we find that 
 $\varprojlim_{n \geq 0} Z_n \subset \varprojlim_{n \geq 0} \Sigma_{E_n} = \Sigma$ 
 (cf.~Lemma~\ref{l:closed-lim}). 
It follows that $x \in \Sigma$. 
Since $d=x\vert_H$, we conclude that $d \in \Sigma_H$ and thus $\Sigma_H$ is closed in $A^H$ 
with respect to prodiscrete topology. 
\end{proof}

Using standard properties of complete algebraic varieties, 
the above proof can be easily adapted to show the following similar result for algebraic subshifts. 
 
\begin{lemma}
\label{l:restriction-closed-complete} 
With the above notations and hypotheses, assume moreover that $\Sigma$ 
satisfies the following condition: 
\begin{enumerate} [\rm (Q)] 
\item 
for every finite subset $E \subset G$,  
$\Sigma_E$ is a complete $K$-algebraic variety and for all finite subset $F \subset E$, 
the projection 
$\pi_{F,E}  \colon \Sigma_E \to \Sigma_F$ is a morphism of $K$-algebraic varieties. 
\end{enumerate}
Then for every subset $H \subset G$, 
the restriction $\Sigma_H$ is also closed in $A^G$ with respect to the prodiscrete topology. 
\qed
\end{lemma}

 \section{Algebraic and algebraic group subshifts} 
\label{s:alg-algr-shift}
We begin with an observation that all algebraic group subshifts (cf.~Definition~\ref{d:alg-subgroup-shift}) 
are automatically abstract subgroups. 
 
\begin{proposition}
\label{p:alg-grp-sshift-subgroup} 
Let $G$ be a countable group and let $V$ be an algebraic group over an algebraically closed field.  
Let $\Sigma$ be an algebraic group subshift of $V^G$. 
Then $\Sigma$ is an abstract subgroup of $V^G$. 
\end{proposition}

\begin{proof}
Since $G$ is countable, we can find an increasing sequence $(E_n)_{n \geq 0}$ 
consisting of finite subsets of $G$ such that $\bigcup_{n \geq 0} E_n = G$. 
\par
Let $\varepsilon \in V$ be the neutral element and we denote multiplicatively the group law of $V$.   
It follows from the definition algebraic group subshifts that for every $n \geq 0$,  the restriction 
$\Sigma_{E_n}$ is an algebraic subgroup of $V^G$. 
Hence, $\varepsilon^{E_n} \in \Sigma_{E_n}$ for every $n \geq 0$. 
Since $\Sigma$ is closed in $V^G$ with respect to the prodiscrete topology, 
we deduce immediately from Lemma~\ref{l:closed-lim} that 
\[
\varepsilon^{G} = (\varepsilon^{E_n})_{n \geq 0} \in \varprojlim_{n \geq 0} \Sigma_{E_n} = \Sigma .
\] 
\par 
Now let $x, y \in \Sigma$. Then clearly $(x^{-1}y)\vert_{E_n} = (x^{-1})\vert_{E_n} y\vert_{E_n}   \in \Sigma_{E_n}$ for every $n \geq 0$. 
Then as above, Lemma~\ref{l:closed-lim} also implies that $x^{-1}y \in \Sigma$. We conclude 
that $\Sigma$ is indeed an abstract subgroup of $V^G$. 
\end{proof}

The next lemma shows that inductions of algebraic group subshifts are also algebraic group subshifts. 

\begin{lemma}
\label{l:alg-shift-induction}
Let $H$ be a subgroup of a group $G$. 
Let $V$ be an algebraic group over an algebraically closed field 
and let $\Lambda$ be an algebraic group subshift of $V^H$. 
Then $\Sigma (V^G ; H, \Lambda)$ is   an algebraic group subshift of $V^G$. 
\end{lemma}

\begin{proof} 
 With respect to the prodiscrete topology, as $\Lambda$ is closed in $A^H$, Lemma~\ref{l:restrict-closed-subshift} 
implies that $\Sigma$ is also closed in $A^G$. 
\par 
As in Lemma~\ref{l:restrict-closed-subshift}, we have $\Lambda = \Sigma_H$ and there is a canonical factorization 
\begin{equation} 
\label{e:subgroup-markov}
\Sigma = \prod_{c \in G/H} \Sigma_c, \quad \quad x \mapsto (x\vert_c)_{c \in G/H}.
\end{equation}
\par 
Let $E \subset G$ be a finite subset. Then \eqref{e:subgroup-markov} induces 
a factorization $\Sigma_E = \prod_{c \in G/H} \Sigma_{c\cap E}$. 
For each $c \in G/H$, choose $h_c \in G$ such that $h_c c= H$. 
Observe that 
\[
h_c^{-1} \Sigma_{c \cap E} = \Sigma_{h_c(c \cap E)} =  (\Sigma_H)_ {h_c(c \cap E)} 
= \Lambda_{h_c(c \cap E)} \subset V^{h_c(c \cap E)}. 
\] 
\par Since $h_c(c \cap E)$ is a finite subset of $H$ and since $\Lambda$ is an algebraic group subshift, 
we deduce that  
$ \Lambda_{h_c(c \cap E)}$ is an algebraic subgroup of $V^{h_c(c \cap E)}$ thus  
$\Lambda_{c \cap E}$ is also an algebraic subgroup of $V^{c \cap E}$.  
This implies that $\Sigma_E = \prod_{c \in G/H} \Sigma_{c\cap E}$ is indeed an algebraic  
subgroup of $V^E$. The conclusion follows. 
\end{proof}

\par 
The following results are direct generalizations  
of \cite[Theorem~10.1]{cscp-2020}, \cite[Theorem~8.1]{cscp-2020}, 
and  \cite[Theorem~7.1]{cscp-2020} 
in the context of algebraic group subshifts over a countable universe. 
 
\begin{theorem} 
\label{t:descending-alg-sft}
Let $G$ be a countable group. 
Let $V$ be an algebraic group over an algebraically closed field 
and let $\Sigma \subset V^G$ be a subshift. 
Consider the following properties: 
\begin{enumerate} [\rm (a)] 
\item 
$\Sigma$ is a subshift of finite type; 
\item 
$\Sigma$ is an algebraic group subshift of finite type; 
\item 
every descending sequence of algebraic group subshifts of $V^G$ 
\begin{equation*} 
\Sigma_0 \supset \Sigma_1 \supset \cdots \supset \Sigma_n \supset \Sigma_{n+1} \supset  \cdots  
\end{equation*} 
such that $\bigcap_{n \geq 0} \Sigma_n = \Sigma$ eventually stabilizes. 
\end{enumerate}
Then we have $\mathrm{(b)}$$\implies$$\mathrm{(a)}$$\implies$$\mathrm{(c)}$. 
Moreover, if $\Sigma \subset A^G$ is an algebraic group subshift, 
then $\mathrm{(a)}$$\iff$$\mathrm{(b)}$$\iff$$\mathrm{(c)}$. 
 \end{theorem} 

\begin{proof}
The prove is the same, \emph{mutatis mutandis}, as the proof given in \cite[Theorem~10.1]{cscp-2020}. 
First, one remarks   that the key local result \cite[Theorem~7.1]{cscp-2020} 
now becomes part of the definition of algebraic group subshifts. 
Second, the only  modification needed to weaken the finite generation condition on the group $G$ in \cite[Theorem~7.1]{cscp-2020} 
by the countability assumption on $G$ 
is the following.  
Since $G$ is countable, we can find an increasing sequence 
$(M_i)_{i \geq 0}$ of finite subsets of $G$ such that $1_G \in M_0$ and 
$\bigcup_{i \geq 0} M_i= G$. Then 
it suffices to replace everywhere the finite subset $M^i$ 
in the proof of \cite[Theorem~10.1]{cscp-2020} 
by the finite subset $M_i$ for every $i \geq 0$. 
\end{proof}

Likewise, the fundamental closed image property of algebraic cellular automaton also holds 
for endomorphisms of algebraic group subshifts. 

\begin{theorem}
\label{t:closed-image}
Let $G$ be a countable group. 
Let $U, V$ be algebraic groups over an algebraically closed field.  
Let $\Sigma$ be an algebraic group subshift of $U^G$ and let 
$\tau \in \Hom_{U,V, G\text{-algr}}(U^G, V^G)$.  
Then $\tau(\Sigma)$ is closed in $V^G$ with respect to the prodiscrete topology.   
\end{theorem}

\begin{proof} 
With a similar modification indicated in the proof of Theorem~\ref{t:descending-alg-sft}, 
the proof of Theorem~\ref{t:closed-image} is, \emph{mutatis mutandis}, the same as the proof of \cite[Theorem~8.1]{cscp-2020}. 
We leave the details to the readers.  
 \end{proof}

A similar and straightforward extension to countable universe of \cite[Theorem~7.1]{cscp-2020} 
together with Theorem~\ref{t:closed-image} give us the following result. 

\begin{theorem}
\label{t:sofic-alg-grp}
Let $G$ be a countable group and let $V$ be an algebraic group over an algebraically field. 
Suppose that $\Sigma$ is an algebraic group sofic subshift of $V^G$. 
Then $\Sigma$ is an algebraic group subshift of $V^G$. 
\qed
\end{theorem}

We now introduce \emph{algebraic subshifts} which 
will be necessary for the statement of Theorem~\ref{density-complete-z}. 
The definition is analogous to Definition \ref{d:alg-subgroup-shift} of algebraic group subshifts given in the Introduction. 

\begin{definition}
\label{d:algebraic-subgroup-shift} 
Let $G$ be a group and let $V$ be an algebraic variety over an algebraically closed field.  
A closed subshift $\Sigma \subset V^G$ is called an \emph{algebraic subshift} of $V^G$ 
if for every finite subset $E \subset G$, 
the restriction $\Sigma_E \subset V^E$ is an   algebraic subvariety.       
\end{definition}
\par 
It is clear that all algebraic group subshifts are algebraic subshifts. 
Moreover, it follows from \cite[Theorem~7.1]{cscp-2020} and Theorem~\ref{t:closed-image}
that whenever all the alphabets involved are complete algebraic varieties over the same algebraically closed field, 
all algebraic sofic subshift and thus all algebraic subshifts of finite type defined in \cite{cscp-2020} are  algebraic subshifts. 
As for algebraic group subshifts discussed above, the universe can be taken to be countable and not necessarily finite generated.  
\par 
Remark that there exists  algebraic subshifts which are not of finite type, e.g., the even subshift (cf.~\cite[Example~3.14]{kurka-book}).  
\par 
We have the following useful lemma which tells us that 
algebraic group subshifts and certain algebraic subshifts satisfy 
the properties (P) and (Q) introduced in  
Lemma~\ref{l:restriction-closed} and in Lemma~\ref{l:restriction-closed-complete}.   
\begin{lemma} 
\label{l:restriction-map-also-alg}
Let $G$ be a group and let $V$ be an algebraic group (resp. a complete algebraic variety) over an algebraically closed field.  
Let $\Sigma$ be an algebraic group subshift (resp. an algebraic subshift) of $V^G$. 
Let $E \subset F$ be finite subsets of $G$. 
Then the restriction to $\Sigma_F$ of the canonical projection $V^F \to V^E$ 
induces a well-defined homomorphism of algebraic groups (resp. morphism of algebraic varieties) 
$\Sigma_F \to \Sigma_E$. 
\end{lemma}

\begin{proof}
The verifications are straightforward using Definition~\ref{d:algebraic-subgroup-shift}.  
We only recall the fact that 
if a map between algebraic groups is a homomorphism of abstract groups which is also a morphism of algebraic varieties then 
it is actually a homomorphism of algebraic groups.  
\end{proof}

\section{Intersection of algebraic group subshifts} 
\label{s:intersection}
It is a straightforward verification that the intersection of finitely many algebraic group subshifts is also an algebraic group subshift. 
When the universe is a countable group, we will show that  the intersection of every descending 
sequence of algebraic group subshifts is again an algebraic group subshift. 

\begin{theorem}
\label{c:descending-property}
Let $G$ be a countable group. 
Let $V$ be an algebraic group over an algebraically closed field.  
Suppose that $(\Sigma_n)_{n \geq 0}$ is a descending sequence of algebraic group subshifts of $V^G$. 
Then $\Sigma \coloneqq \bigcap_{n \geq 0} \Sigma_n$ is an algebraic group subshift of $V^G$. 
\end{theorem}

\begin{proof} 
Since $\Sigma_n$ is a closed subshift of $V^G$ for all $n \geq 0$, 
so is the intersection $\Sigma$. 
Let $E \subset G$ be a finite subset. 
By definition,  $(\Sigma_n)_E \subset V^E$ is an algebraic subgroup  
for every $n \geq 0$. 
We clearly have $\Sigma_E \subset (\Sigma_n)_E$ for every $n \geq 0$ so   $\Sigma_E \subset \bigcap_{n \geq 0} (\Sigma_n)_E$. 
\par 
For the converse conclusion, let $z \in \bigcap_{n \geq 0} (\Sigma_n)_E$.  
Since $G$ is countable, 
we can find an increasing sequence  $(M_n)_{n \geq 0}$ of finite subsets of $G$ such that $\{1_G\} \in M_0$ 
and $\bigcup_{n \geq 0} M_n =G$. 
\par 
Consider the inverse system $(\Sigma_{ij})_{i,j \geq 1}$ defined by $\Sigma_{ij} \coloneqq (\Sigma_j)_{M_{i}} \subset V^{M_{i}}$ 
for every $i,j \geq 1$. 
Note that $\Sigma_{i,j+1} \subset \Sigma_{ij}$ since $\Sigma_{j+1} \subset \Sigma_j$  for all $i, j \geq 1$. 
Moreover, every $\Sigma_{ij}$ is an algebraic subgroup of $V^{M_{i}}$ since every $\Sigma_j$ is an 
 algebraic group subshift of $V^G$. 
 \par 
The \emph{unit horizontal transition maps} of the inverse system are the canonical homomorphisms of algebraic groups  
$p_{ij} \colon \Sigma_{i+1, j} \to \Sigma_{ij}$ defined by  
$p_{ij}(x)= x\vert_{M_{i}}$ for every $x \in \Sigma_{i+1, j}$ as $M_{i} \subset M_{i+1}$ 
(cf.~Lemma~\ref{l:restriction-map-also-alg}).  
The \emph{vertical unit transition maps} $q_{ij} \colon \Sigma_{i, j+1} \to \Sigma_{ij}$ are simply defined as the inclusion homomorphisms. 
\par 
The directed set $I = \{1, 2, \cdots\}^2$ is partially ordered by $(u,v) \prec (i,j)$ if and only if $u \leq i$ and $v \leq j$.  
The transition maps of the inverse system $(\Sigma_{ij})_{(i,j) \in I}$ 
are the compositions of the unit transition maps and they are well-defined as it can be checked 
easily that for all $(i,j) \in I$: 
\begin{align*}
q_{i j} \circ p_{i,j+1} = p_{i j} \circ q_{i+1,j} && \mbox{(see also \cite[Section~4]{cscp-2020}). }
\end{align*} 
\par  
The maps $p_{ij} \colon (\Sigma_j)_{M_{i+1}} \to (\Sigma_j)_{M_{i}}$ 
are the canonical homomorphisms of algebraic groups 
induced by the inclusions $M_{i} \subset M_{i+1}$ (Lemma~\ref{l:restriction-map-also-alg}).  
Likewise, the inclusion maps $q_{ij}$ are clearly homomorphisms of algebraic subgroups of $V^{M_{i}}$. 
We deduce that the transition maps of the inverse system $(\Sigma_{ij})_{(i,j) \in I}$ are   homomorphisms of algebraic groups. 
\par 
Consider the inverse subsystem $(X_{ij})_{(i,j) \in I}$ of $(\Sigma_{ij})_{(i,j) \in I}$ defined by 
\[ 
X_{ij} \coloneqq \{ x \in (\Sigma_j)_{M_i} \colon x\vert_E=z  \} \subset \Sigma_{ij}. 
\] 
Then by a similar argument as the proof of \eqref{e:y-k-y-translate} in Lemma~\ref{claim:1}, 
every $X_{ij}$ is a translate of an algebraic subgroup of $\Sigma_{ij}$ for every $(i,j) \in I$. 
Hence, it follows from the above paragraph that the transition maps of the inverse subsystem $(X_{ij})_{(i,j) \in I}$ have Zariski closed images.  
 \par 
Note that since $z \in \bigcap_{n \geq 0} (\Sigma_n)_E$, we have $z\vert_{M_{i}} \in (\Sigma_j)_{M_{i}}= X_{ij}$  
so that every $X_{ij}$ is nonempty for all $(i, j) \in I$. 
Therefore, Lemma~\ref{l:inverse-limit-alg-grp} applied to the 
inverse subsystem $(X_{ij})_{(i,j) \in I}$ tells us that    
$\varprojlim_{(i,j) \in I} X_{ij}$ is nonempty. 
\par 
For every fixed $j \geq 1$, 
Lemma~\ref{l:closed-lim} implies that  
\begin{equation} 
\label{e:descending-1} 
\varprojlim_{i \geq 1} X_{ij} \subset \varprojlim_{i \geq 1} \Sigma_{ij} = \Sigma_j. 
\end{equation}
Hence, we can easily deduce that   

\begin{align*}
\varprojlim_{(i,j) \in I} X_{ij} & = \varprojlim_{j \geq 1} \varprojlim_{i \geq 1} X_{ij}  && \mbox{(cf. also~\cite[relation~(4.7)]{cscp-2020})}\\
& \subset \varprojlim_{j \geq 1} \Sigma_{j}  && \mbox{(by \eqref{e:descending-1})}\\
& =  \bigcap_{j \geq 1} \Sigma_j && \mbox{(by basic properties of inverse limits)} \\
& \eqqcolon \Sigma && \mbox{(by definition of $\Sigma$)}. 
\end{align*}
It follows that there exists $y \in \Sigma$ such that 
$y \in \varprojlim_{(i,j) \in I} X_{ij}$.  
It is immediate from the construction that $y\vert_E = z$. 
Therefore, $\bigcap_{n \geq 0} (\Sigma_n)_E \subset  \Sigma_E$ and we have   
$\Sigma_E = \bigcap_{n \geq 1} (\Sigma_n)_E \subset V^E$.  
\par 
 Since $((\Sigma_n)_E)_{n \geq 1}$ is a descending sequence of algebraic subgroups of $V^E$, it eventually stabilizes 
 by the Noetherianity of the Zariski topology.  
 Hence,  $\Sigma_E $ is also an algebraic subgroup  of $V^E$. 
 This proves  that $\Sigma$ is indeed an algebraic group subshift of $V^G$. 
 The proof of Theorem~\ref{c:descending-property} is completed. 
 \end{proof}

\section{Groups of algebraic Markov type} 
\label{s:markov-type-alg}

In this section, the equivalence between the descending chain condition for the full shifts and the 
finite type property of algebraic group subshifts is established when the universe is countable (Proposition~\ref{c:markov-alg}). 
This leads us to define \emph{groups of algebraic Markov type} (Definition~\ref{d:markov-type-alg}) as an extension 
of the class of \emph{groups of Markov type} \cite[Definition~4.1]{schmidt-book}. 

\subsection{Descending chain condition and finite type property} 

 As an application of Theorem~\ref{c:descending-property}, we obtain the following result 
 similar to \cite[Theorem~3.8]{schmidt-book}: 

\begin{proposition} 
\label{c:markov-alg}
Let $G$ be a countable group and let $V$ be an algebraic group over an algebraically closed field. 
The following are equivalent: 
\begin{enumerate} [\rm (a)] 
\item 
every algebraic group subshift of $V^G$ is an algebraic group subshift of finite type;  
\item 
every descending sequence of algebraic group subshifts of $V^G$ eventually stabilizes.
\end{enumerate} 
 \end{proposition}
 
 \begin{proof} 
Suppose first that (a) is verified. 
Let $(\Sigma_n)_{n \geq 0}$ be a descending sequence of algebraic group subshifts of $V^G$. 
Let $\Sigma \coloneqq \bigcap_{n \geq 0} \Sigma_n$. 
By Theorem~\ref{c:descending-property}, we know that $\Sigma$ is an algebraic group subshift of $V^G$. 
The condition (a) then implies that $\Sigma_n$ is an algebraic group subshift of finite type of $V^G$. 
By Theorem~\ref{t:descending-alg-sft}, it follows that the sequence $(\Sigma_n)_{n \geq 0}$ eventually stabilizes. 
This shows that (a)$\implies$(b). 
\par
Conversely, assume that (b) is satisfied. Let $\Sigma$ be an algebraic group subshift of $V^G$. 
As $G$ is countable, there exists an increasing sequence $(E_n)_{n \geq n}$ 
 of finite subsets of $G$ such that $G= \bigcup_{n \geq 0} E_n$. 
 \par 
Then for every $n \geq 0$, we find that $\Sigma_{E_n}$ is an algebraic subgroup of $V^{E_n}$ 
thus $\Sigma_n \coloneqq \Sigma(V^G; E_n, \Sigma_{E_n})$ is an algebraic group subshift of finite type of $V^G$. 
\par 
It is clear that $ \Sigma \subset \Sigma_{n+1} \subset \Sigma_n$ for every $n \geq 0$.  
In particular, $\Sigma \subset \bigcap_{n \geq 0} \Sigma_n$. 
On the other hand, let $z \in \bigcap_{n \geq 0} \Sigma_n$. 
Then by definition of the subshifts of finite type $\Sigma_n$, 
we have $z \vert_{E_n} \in \Sigma_{E_n}$ for every $n \geq 0$. 
By Lemma~\ref{l:restriction-map-also-alg}, 
we can regard $(\Sigma_{E_n})_{n \geq 0}$ as an inverse system whose transition maps are given by the canonical 
homomorphisms $\Sigma_{E_m} \to \Sigma_{E_n}$ for every $m \geq n \geq 0$. 
Since $\Sigma$ is closed, Lemma~\ref{l:closed-lim} implies that  
$z \in \varprojlim_{n \geq 0} \Sigma_{E_n} = \Sigma$ and hence 
$\bigcap_{n \geq 0} \Sigma_n \subset \Sigma$. 
\par 
We conclude that $\Sigma = \bigcap_{n \geq 0} \Sigma_n$. 
By Theorem~\ref{t:sofic-alg-grp}, we know that $\Sigma_n$ is an algebraic group subshift of $V^G$ for every $n \geq 0$.    
But then (b) implies that the descending sequence $(\Sigma_n)_{n \geq 0}$ 
must stabilize. Hence, there exists $N \geq 0$ 
such that  $\Sigma = \bigcap_{n \geq 0} \Sigma_n = \Sigma_N$. 
Finally, Theorem~\ref{t:descending-alg-sft} implies that $\Sigma_N$ is an algebraic group subshift of finite type of $V^G$ 
and thus so is $\Sigma$. Hence (b)$\implies$(a) and the proof is completed. 
 \end{proof}

\subsection{The class of groups of algebraic Markov type} 

By analogy with the definition of groups of Markov type given in  \cite[Definition~4.1]{schmidt-book},   
we introduce the class of \emph{groups of algebraic Markov type} as follows. 

\begin{definition}
\label{d:markov-type-alg}
A countable group $G$ said to be of \emph{algebraic Markov type} 
if for every algebraic group $V$ over an algebraically closed field, the full shift $V^G$ satisfies 
one of the equivalent conditions of Proposition~\ref{c:markov-alg}. 
\end{definition}
\par 
By the Noetherianity of the Zariski topology, it is clear from the above definition 
that every finite group is a group of algebraic Markov type. 
We are going to show  that the class of groups of algebraic Markov type is stable under taking subgroups. 
 
 \begin{proposition}
\label{p:markov-subgroup} 
Let $G$ be a   group of algebraic Markov type. 
Then every subgroup of $G$ is also a group of algebraic Markov type. 
\end{proposition}

\begin{proof}
Let $H$ be a subgroup of $G$ and let $V$ be an algebraic group over an algebraically closed field.  
Let $(\Lambda_n)_{n \geq 0}$ be a descending sequence 
of algebraic group subshifts of $V^H$. 
\par 
For every $n \geq 0$, we consider the induction subshift $\Sigma_n \coloneqq \Sigma(V^G; H, \Lambda_n)$ 
which is an algebraic group subshift of $V^G$ by Lemma~\ref{l:alg-shift-induction}. 
Since $\Lambda_{n+1} \subset \Lambda_{n}$, 
it follows that $\Sigma_{n+1} \subset \Sigma_n$ for every $n \geq 0$. 
Hence, we obtain a descending sequence $(\Sigma_n)_{n \geq 0}$ 
of algebraic group subshifts of $V^G$. 
\par 
Since $G$ is a group of algebraic Markov type by hypothesis, 
the sequence $(\Sigma_n)_{n \geq 0}$ eventually stabilizes. 
Since $(\Sigma_n)_H= \Lambda_n$ for every $n \geq 0$ by Lemma~\ref{l:restrict-closed-subshift}, 
it follows that the sequence $( \Lambda_n)_{n \geq 0}$ must also stabilize. 
By Proposition~\ref{c:markov-alg} and Definition~\ref{d:markov-type-alg}, 
we can thus conclude that $H$ is a group of algebraic Markov type. 
\end{proof}

In Section~\ref{s:main}, we will see that the extensions of cyclic groups by groups of algebraic Markov type are     groups of algebraic Markov type 
(cf.~Theorem~\ref{t:algebraic-group-Z-d} and Proposition~\ref{p:main-by-finite}).

\section{Dynamical finiteness of algebraic group subshifts} 
\label{s:main} 

Our goal in this section is to give a proof of the following  result which is analogous to 
\cite[Theorem~4.2]{schmidt-book}.  

\begin{theorem}
\label{t:main-polycyclic-finite} 
Every polycyclic-by-finite group is of algebraic Markov type. 
\end{theorem} 

By the definition of groups of algebraic Markov type, we deduce immediately the main theorems mentioned in the Introduction.   

\begin{proof}[Proof of Theorem~\ref{t:intro-descending-property-alg-grp} and Theorem~\ref{t:intro-sft-property-alg-grp}]
By Definition~\ref{d:markov-type-alg}, 
it is a direct consequence of Theorem~\ref{t:main-polycyclic-finite} and Proposition~\ref{c:markov-alg}. 
\end{proof}

The proof of Theorem~\ref{t:main-polycyclic-finite} will occupy the rest of the present section. 
It results from Theorem \ref{t:algebraic-group-Z-d} in Section~\ref{s:main-infinite-cyclic} 
and Proposition~\ref{p:main-by-finite} in  Section~\ref{s:main-finite-cyclic} below. 

\subsection{The case of infinite cyclic extension} 
\label{s:main-infinite-cyclic}
We will now prove the main technical result of the paper which is an extension 
of \cite[Lemma~4.4]{schmidt-book}. We remark that the proof of \cite[Lemma~4.4]{schmidt-book} 
relies in a crucial way on the compactness of the alphabets and thus the compactness of the induced Tychonoff topology on the full shifts. 
However, in our setting, the full shift is equipped with the prodiscrete topology and is never compact unless  when 
the underlying alphabet  is finite. 

\begin{theorem} 
\label{t:algebraic-group-Z-d} 
Let $0 \to H \to G  \xrightarrow{\varphi} \Z \to 0$ be an extension of countable groups. 
Suppose that $H$ is a group of algebraic Markov type. 
Then the group $G$ is also of algebraic Markov type. 
\end{theorem} 

\begin{proof} 
Let $V$ be an algebraic group over an algebraically closed field. 
Let $\Sigma \subset V^G$ be an algebraic group subshift. 
We must show that $\Sigma \subset V^G$ is an algebraic group subshift of finite type.   
 \par 
 We denote by $\varepsilon$  the neutral element of $V$. 
The group laws on $V$ is written multiplicatively. 
Let $0_H = 0_G$ be the neutral element of the groups $H$ and $G$ whose group laws are denoted 
additively as the group of integers $\Z$.  
\par 
Choose an arbitrary element $a \in G$ such that $\varphi(a)=1$. Such $a$ exists since $\varphi$ is surjective. 
Then we have a decomposition of $G$ into disjoint cosets of $H$ in $G$:  
\begin{align}
\label{e:decompostion-coset-z}
G = \coprod_{n \in \Z} Ha^n,
\end{align}
which defines a bijection $\Phi \colon H \times \Z \to G$ given by $\Phi((h,n))= h a^n$ 
for every $h \in H$ and $n \in \Z$. We regard $\Phi$ as a coordinate function for $G$. 
\par 
Since $H$ is countable, we can find an increasing sequence $(F_n)_{n \geq 1}$ 
of finite subsets of $H$ such that $0_H \in F_1$ and $H = \bigcup_{n \geq 1} F_n$. 
\par 
For every integer $n \geq 1$, let us denote $G_n \coloneqq  \{-n, \dots, 0 \} \subset \Z$  and 
$G_n^* \coloneqq \{-n, \dots, -1 \} \subset \Z$ and $I_n \coloneqq \{ -n, \cdots, n \} \subset \Z$.  
We define   
\begin{align}
\label{e:def-x-n}
X_n  \coloneqq \{x\vert_H \colon x \in \Sigma , \, x\vert_{\Phi(H \times G_n^*)} = \varepsilon^{\Phi(H \times G_n^*)}\} 
\subset   V^H.  
\end{align}
Hence, we can regard $X_n$ as a subset of $V^H$. 
It is straightforward from the definition that $X_{n+1} \subset X_n$ and  $X_n$ is $H$-invariant for every $n \geq 1$. 
\par 
For ease of reading, we divide the proof of Theorem~\ref{t:algebraic-group-Z-d} into several lemmata. 

\begin{lemma} 
\label{claim:1}
Let $E \subset H$ be a finite subset. Then $(X_n)_E$  
is an algebraic subgroup of $V^E$ for every $n \geq 1$. 
\end{lemma} 
\begin{proof} 
Fix an integer $n \geq 1$. Since $E$ is finite, there exists an integer $k_0 \geq n$ such that $E \subset F_{k_0}$.  
Consider the following subset $Y \subset \Sigma_E$ defined by:  
\begin{align}
\label{e:y-claim-1}
Y \coloneqq \bigcap_{k \geq k_0}  Y_k , \quad \mbox{where} \quad 
Y_k \coloneqq  \{ x\vert_E \colon x \in \Sigma, \, x\vert_{\Phi(F_k \times G_n^*)} = \varepsilon^{\Phi(F_k \times G_n^*)}  \}.  
\end{align}
\par 
Let $k \geq k_0$ be an integer. Since $\Sigma$ is an algebraic group subshift, 
$\Sigma_{\Phi(F_k \times G_n)}$ is an algebraic subgroup of $V^{\Phi(F_k \times G_n)}$. 
Let $\pi^* \colon \Sigma_{\Phi(F_k \times G_n)} \to \Sigma_{\Phi(F_k \times G_n^*)}$ 
and $\pi_E \colon \Sigma_{\Phi(F_k \times G_n)} \to \Sigma_E$  
be the canonical homomorphisms induced respectively by the inclusions 
$\Phi(F_k \times G_n^*) \subset \Phi(F_k \times G_n)$ and $E \subset \Phi(F_k \times G_n)$  
(cf.~Lemma~\ref{l:restriction-map-also-alg}).  
\par 
Then we find that $Y_k = \pi_E(\Ker(\pi^*))$ which is clearly an algebraic subgroup of $\Sigma_E$. 
Therefore, \eqref{e:y-claim-1} implies that 
$Y$ is the intersection of a descending sequence of algebraic subgroups of $\Sigma_E$ 
and thus of $V^E$.  
By the Noetherianity of the Zariski topology on $V^E$, it follows that 
$Y$ is an algebraic subgroup of $V^E$.  
\par 
 We are going to prove that  $(X_n)_E=Y$. 
The inclusion $(X_n)_E \subset Y$  
 is immediate. For the converse inclusion, let 
 $y \in Y$. We must show that there exists $x \in X_n$ such that $x\vert_E = y$.  
For every $k \geq k_0$,   consider the following subset of $\Sigma_{\Phi( F_k \times I_k)}$:   
\begin{align}
\label{e:y-k-y}
Y_k(y) \coloneqq  \{ x\vert_{\Phi( F_k \times I_k)} \colon x \in \Sigma, 
\, x\vert_E =  y, \, x\vert_{\Phi(F_k  \times G_n^*)} = \varepsilon^{\Phi(F_k \times G_n^*)} \}. 
\end{align} 
\par 
Since $y \in Y = \bigcap_{k \geq k_0}  Y_k$, we can find for every $k \geq k_0$  
a configuration $x_k \in \Sigma$ such that $y= x_k \vert_E$ and 
$x_k\vert_{\Phi(F_k \times G_n^*)} = \varepsilon^{\Phi(F_k \times G_n^*)}$. 
This shows that $x_k\vert_{\Phi(F_k \times I_k)} \in Y_k(y)$ for every $k \geq k_0$. 
\par 
For every integer $k \geq k_0$, consider the canonical  homomorphisms of algebraic groups 
$\psi_k^* \colon \Sigma_{\Phi(F_k\times I_k)} \to \Sigma_{\Phi(F_k \times G_n^*)}$ 
and $\phi_k \colon \Sigma_{\Phi(F_k\times I_k)} \to \Sigma_E$ (cf.~Lemma~\ref{l:restriction-map-also-alg}). 
Then it is not hard to see that 
\begin{align}
\label{e:y-k-y-translate}
Y_k(y) =  x_k\vert_{\Phi( F_k \times I_k)}  ( \Ker (\psi_k^*) \cap \Ker(\phi_k)).
\end{align}
\par 
Hence, we find that $Y_k(y)$ is a translate by $ x_k\vert_{\Phi( F_k \times I_k)} $ of the algebraic subgroup 
$\Ker (\psi_k^*) \cap \Ker(\phi_k)$ of $\Sigma_{\Phi(F_k\times I_k)}$ for every $k \geq k_0$.  
\par 
We have seen that the sequence $(Y_k(y))_{k \geq k_0}$ 
forms an inverse system of nonempty sets. The  
transition maps $Y_m(y) \to Y_k(y)$, 
where $m \geq k \geq k_0$, are the restrictions of the canonical homomorphisms  
$\Sigma_{\Phi(F_m\times I_m)} \to \Sigma_{\Phi(F_k\times I_k)}$ induced by the inclusions $\Phi(F_k \times I_k) \subset \Phi(F_m\times I_m)$. 
\par 
Since $Y_m(y)$ is a translate of an algebraic subgroup of $\Sigma_{\Phi(F_m \times I_m)}$, 
the transition maps $Y_m(y) \to Y_k(y)$ have Zariski closed images for all $m \geq k \geq k_0$. 
\par 
Therefore, Lemma~\ref{l:inverse-limit-alg-grp} implies that there exists 
$x \in \varprojlim_{k \geq k_0} Y_k(y)$. 
By the construction of the set $Y_k(y)$, we find that $x\vert_E = y$ 
and that $x\vert_{\Phi(F_k \times G_n^*)} = \varepsilon^{\Phi(F_k \times G_n^*)}$ for every $k \geq k_0$.   
\par 
Since $H = \bigcup_{k \geq k_0} F_k$, 
we deduce that 
$x\vert_{\Phi(H \times G_n^*)} = \varepsilon^{\Phi(H \times G_n^*)}$.  
\par 
Note that $ \varprojlim_{k \geq k_0} Y_k(y) \subset \varprojlim_{k \geq n} \Sigma_{\Phi(F_k\times I_k)}$ 
since $Y_k(y) \subset \Sigma_{\Phi(F_k\times _k)}$ for every $k \geq k_0$.  
On the other hand, $\varprojlim_{k \geq n} \Sigma_{\Phi(F_k\times I_k)} = \Sigma$ 
by the closedness of $\Sigma$ in $V^G$ with respect to the prodiscrete topology 
and as $\bigcup_{k \geq n} \Phi(F_k\times I_k)= G$ 
(cf.~Lemma~\ref{l:closed-lim}). 
It follows that $x \in \Sigma$. 
\par 
Hence, by definition of $X_n$ (cf.~\eqref{e:def-x-n})), 
it follows that $x \in X_n$. 
Since $x\vert_E = y$ as well, we deduce that $Y \subset (X_n)_E$.  
\par 
We can thus conclude that  $(X_n)_E = Y$ is an algebraic subgroup 
of $V^E$. This proves Lemma~\ref{claim:1}. 
\end{proof}

\begin{lemma}
\label{claim:2}
For every integer $n \geq 1$, the subshift $X_n$ is closed in $V^H$ 
with respect to the prodiscrete topology.  
\end{lemma} 
\begin{proof}
Let us fix an integer $n \geq 1$.  
 For every $k \geq n \geq 1$, let 
 \[
X_{nk} \coloneqq \{x\vert_{\Phi(F_k \times I_k)} \colon x \in \Sigma , \, x\vert_{\Phi(F_k \times G_n^*)} = \varepsilon^{\Phi(F_k \times G_n^*)}\} 
\subset \Sigma_{\Phi(F_k \times I_k)}. 
\]
In other words, $X_{nk}$ is the kernel of 
 the canonical homomorphism of algebraic groups $\Sigma_{\Phi(F_k \times I_k)} \to \Sigma_{\Phi(F_k \times G_n^*)}$ 
 (cf.~Lemma~\ref{l:restriction-map-also-alg}). 
It follows that $X_{nk}$ is  an algebraic subgroup  of $\Sigma_{\Phi(F_k \times I_k)}$ and thus of $V^{\Phi(F_k \times I_k)}$.  
\par
For all $m \geq k \geq n$, the inclusion $\Phi(F_k \times I_k) \subset \Phi(F_m \times I_m)$ 
induces a projection $\pi_{km} \colon V^{ \Phi(F_m \times I_m)} \subset V^{\Phi(F_k\times I_k)}$. 
If $x \in V^{ \Phi(F_m \times I_m)}$ satisfies $x\vert_{\Phi(F_m \times G_n^*)} = \varepsilon^{\Phi(F_m \times G_n^*)}$ 
then clearly  $\pi_{km}(x)\vert_{ \Phi(F_k \times G_n^*)} = \varepsilon^{\Phi(F_k \times G_n^*)}$. 
Hence, the restriction of $\pi_{km}$ to $X_{nm}$ defines a homomorphism of algebraic groups 
$p_{km} \colon X_{nm} \to X_{nk}$. 
\par 
We thus obtain an inverse system $(X_{nk})_{k \geq n}$ whose transition maps $p_{km}$ 
are homomorphisms of algebraic groups for $m \geq k \geq n$.   
\par 
Now suppose that $z \in V^H$ belongs to the closure of $X_n$ in $V^H$ with respect to the 
prodiscrete topology. 
Hence, by definition of $X_n$, there exists for each $k \geq n$ 
a configuration   
$y_k \in \Sigma$ such that $ y_k \vert_{F_k} = z\vert_{F_k}$ 
and that $y_k \vert_{\Phi(H \times G_n^*)} = \varepsilon^{\Phi(H \times G_n^*)}$. 
\par 
For every $k \geq n$, consider the following subset of $\Sigma_{\Phi(F_k\times I_k)}$: 
 \begin{equation}
 \label{e:x-n-k-1}
X_{nk}(z) \coloneqq \{ x\vert_{\Phi(F_k\times I_k)} \colon 
 x \in \Sigma , \, x\vert_{\Phi(F_k \times G_n^*)} = \varepsilon^{ \Phi(F_k \times G_n^*)}, \, 
 x\vert_{F_k} = z \vert_{F_k}  \}. 
\end{equation}
\par 
Observe that $y_k\vert_{\Phi(F_k\times I_k)} \in X_{nk}(z)$ for every $k \geq n$. 
A similar argument as in the proof of \eqref{e:y-k-y-translate} shows that $X_{nk}(z)$ is a translate of an algebraic subgroup of $\Sigma_{\Phi(F_k\times I_k)}$. 
\par 
Therefore, we obtain an inverse subsystem $(X_{nk}(z))_{k \geq n} $ of $(X_{nk})_{k \geq n}$ 
where every  $X_{nk}(z)$ is nonempty for $k \geq n$ and all the transition maps have Zariski closed images.  
Then, again by Lemma~\ref{l:inverse-limit-alg-grp},  
there exists $x \in \varprojlim_{k \geq n} X_{nk}(z) \subset \varprojlim_{k \geq n} \Sigma_{\Phi(F_k\times I_k)} = \Sigma$. 
By construction, we find that 
$x\vert_{\Phi(F_k \times G_n^*)} = \varepsilon^{\Phi(F_k \times G_n^*)}$ 
and that $x\vert_{F_k} = z \vert_{F_k}$ for every $k \geq n$. 
Thus, by letting $k \to \infty$, we obtain 
$x\vert_{\Phi(H \times G_n^*)} = \varepsilon^{\Phi(H \times G_n^*)}$ and 
$x\vert_H = z$. 
Hence, $z= x\vert_H  \in X_n$ and this proves that $X_n$ is closed  in $V^H$ 
with respect to the prodiscrete topology.  
\par 
The proof of Lemma~\ref{claim:2} is completed. 
 \end{proof}

Combining Lemma~\ref{claim:1},  Lemma~\ref{claim:2} and the observations made 
after the definition \eqref{e:def-x-n} of $X_n$, we obtain the following:  

\begin{lemma} 
\label{l:lemma-1}
The sequence $(X_n)_{n \geq 1}$ is a descending sequence of   algebraic group subshifts of $V^H$. 
\qed
\end{lemma}  

We can now complete the proof of Theorem~\ref{t:algebraic-group-Z-d} as follows. 
Since $H$ is of algebraic Markov type, 
Lemma~\ref{l:lemma-1}  and Theorem~\ref{t:descending-alg-sft} imply that the descending sequence $(X_n)_{n \geq 1}$ 
must stabilize and consist of algebraic group subshifts of finite type of $V^H$. 
In particular,  there exists $N \geq 1$ such that $X_n = X_N \eqqcolon X$ for every $n \geq N$. 
\par 
For the notations, we denote $\Z_{<n} = \{x \in \Z \colon x < n \}$  for every $n \in \Z$.  
Then we find that 
\begin{equation}
\label{e:proof-main-X} 
X = \bigcap_{n \geq 1} X_n 
= \{x\vert_H \colon x \in \Sigma , \, x\vert_{\Phi(H \times \Z_{<0})} = \varepsilon^{\Phi(H \times \Z_{<0})} \}.  
\end{equation} 
Remark that $X$ is nonempty since it clearly contains $0^H$. 
It follows immediately from \eqref{e:proof-main-X} that
\begin{align*}
L 
& \coloneqq 
\{ x \in \Sigma \colon x\vert_{\Phi(H \times \Z_{<0})} = \varepsilon^{\Phi(H\times  \Z_{<0})}, \, x\vert_H \in X \} \\
& = 
\{ x \in \Sigma \colon x\vert_{\Phi(H \times \Z_{<0})} = \varepsilon^{\Phi(H \times  \Z_{<0})} \}.  
\end{align*} 
For every $v \in X$, consider the following subset  of $\Sigma$  
\[
L(v) \coloneqq  \{ x \in \Sigma \colon x\vert_{ \Phi(H \times \Z_{<0})} = \varepsilon^{\Phi(H \times  \Z_{<0})}, \, 
x\vert_H = v \}. 
\]
Then by the definition of $L(v)$ and the relation $\eqref{e:proof-main-X}$, 
it is clear that $L(v)$ is nonempty for every $v \in X$. 
\par 
 Let $\Omega \coloneqq \Phi(H \times I_N) \subset G$. 
Consider the subshift 
$\Sigma' \coloneqq  \Sigma (V^G; \Omega, \Sigma_\Omega)$ of $V^G$ (see Definition \eqref{e:sft}). 
 It is clear that $\Sigma \subset \Sigma'$. 
We are going to prove the converse inclusion.  
\par 
 Let $y \in \Sigma'$ be a configuration.  
 Note that  $a  = \Phi( (0_H, 1) )\in G$ and let $g = (N+1)a = \Phi( (0_H, N+1) ) \in G$. 
 Then by definition of $\Sigma'$, there exists $z_0 , z_1 \in \Sigma$ 
such that $(z_0)\vert_\Omega = y\vert_\Omega$ 
and  $(z_1)\vert_{a  + \Omega} = y\vert_{a + \Omega}$. 
It follows that for $z = z_0(z_1)^{-1} \in \Sigma$, we have 
$z\vert_{\Phi(H \times \{1, \cdots, N \})} = \varepsilon^{ \Phi(H \times \{ 1, \cdots, N \} )}$. 
Hence,  $v \coloneqq ((-g)z)\vert_H \in X$. 
\par 
 Let $c \in L(v)$. 
Since the subshift $\Sigma$ is $G$-invariant, we have $gc \in \Sigma$ 
and the configuration $x \coloneqq (gc)^{-1}z_0 \in \Sigma$ satisfies 
\[
x\vert_{\Phi( H \times \{0, \cdots, N+1 \})} = y\vert_{ \Phi( H \times \{0, \cdots, N+1 \})}. 
\]  
An immediate induction on $m \geq 1$ by a similar argument shows that 
there exists a sequence $(x_m)_{m \geq 1} \subset \Sigma$ such that 
$x_m\vert_{\Phi( H \times \{0, \cdots, m \})}= y\vert_{\Phi(H \times \{0, \cdots, m \})}$ for every $m \geq 1$. 
\par 
Remark that any given finite subset of $G$ is contained in 
some translate of the sets $\Phi(H  \times \{0, \cdots, m\})$ for some $m \geq 1$. 
Consequently, the above paragraph shows  that $y$ belongs to the closure of $\Sigma$ in $V^G$ 
with respect to the prodiscrete topology.  
As $\Sigma$ is closed in $V^G$, it follows that $y \in \Sigma$. 
Therefore, $\Sigma' \subset \Sigma$ and we conclude that $\Sigma = \Sigma' \subset V^G$.  
\par 
We regard $\Sigma_\Omega$ as a subshift of $U^H$ with respect to the shift action given by the group $H$ 
with the alphabet 
$U \coloneqq V^{ \Phi(\{0_H\} \times I_N)}$ which is clearly an algebraic group. 
\par 
As $\Sigma$ is closed in $V^G$ with respect to the prodiscrete topology, 
it follows from Lemma~\ref{l:restriction-closed} 
that $\Sigma_\Omega$ is also closed in $U^H$. 
Moreover, as $\Sigma$ is an algebraic group subshift, we deduce that 
$(\Sigma_\Omega)_{\Phi( E \times I_N)} \subset U^E$ is an algebraic subgroup for every finite subset $E \subset H$. 
Therefore, $\Sigma_\Omega$ is an algebraic group subshift of $U^H$. 
\par 
Since $H$ is a group of algebraic Markov type, $\Sigma_\Omega$ is an algebraic group subshift of finite type of $U^H$. 
Hence,  there exists a finite subset $D \subset H$ such that 
$\Sigma_\Omega = \Sigma(U^H; D, P)$ 
where $P \coloneqq (\Sigma_\Omega)_{\Phi( D \times I_N)} = \Sigma_{ \Phi(D \times I_N)}$ 
is  an algebraic subgroup of $U^D=V^{\Phi(D \times I_N)}$.  
\par 
Finally, it is straightforward from the above that 
\begin{align*}
\Sigma & = \Sigma'= \Sigma(V^G; \Omega, \Sigma_\Omega) &&\\
& = \Sigma(V^G; \Phi(H \times I_N), \Sigma(U^H; D, P)) &&\\
& = \Sigma(V^G; \Phi(D \times I_N), P) &&\mbox{(by Lemma~\ref{l:restriction-sft})}.
\end{align*} 
Since $\Phi( D \times I_N)$ is finite and $P$ is an algebraic subgroup of $V^{\Phi( D \times I_N)}$, 
we conclude that $\Sigma$ is an algebraic group subshift of finite type of $V^G$.  
\par 
The proof  of Theorem~\ref{t:algebraic-group-Z-d} is complete. 
\end{proof}

\subsection{The case of extension by finite groups} 
\label{s:main-finite-cyclic} 

The following proposition is a direct application of Lemma~\ref{l:restriction-sft}.  

\begin{proposition}
\label{p:main-by-finite} 
 Let $0 \to H \to G  \xrightarrow{\varphi} F \to 0$ be an extension of countable groups. 
Suppose that $H$ is a group of algebraic Markov type and $F$ is finite. 
Then the group $G$ is also of algebraic Markov type. 
 \end{proposition}

\begin{proof}
Let $V$ be an algebraic group over an algebraically closed field 
and let $\Sigma$ be an algebraic group subshift  of $V^G$. We must show that 
$\Sigma$ is an algebraic group subshift of finite type of $V^G$. 
\par 
Since $\varphi$ is surjective, we can choose a finite subset $E \subset G$ 
such that $|E| = |F|$ and $\varphi(E)=F$. Equivalently, $E$ is a complete set of 
representatives of cosets of $H$ in $G$. 
\par 
Let $U= V^E$ then $U$ is an algebraic group. Then by the discussion before Lemma~\ref{l:restriction-sft}, 
we have a canonical bijection $U^H = V^G$: to each $x \in U^H$, we associate an element $y \in V^{HE}$ given  
by $y(h k) \coloneqq (x(h))(k)$ for every $h \in H$ and $k \in E$. 
It is clear that the above bijection commutes with 
the shift actions of the group $H$.  
\par 
Therefore, $\Sigma$ can be regarded as an algebraic group subshift of $U^H$ with the shift action of the group $H$. 
Since $H$ is of Markov type, we deduce that $\Sigma$ is an algebraic group subshift of finite type of $U^H$. 
Thus, there exists a finite subset $D \in H$ and an algebraic subgroup $P \subset U^D$ such that 
$\Sigma = \Sigma(U^H; D, P)$. 
By Lemma~\ref{l:restriction-sft}, we find that 
\begin{equation*}
\Sigma = \Sigma( V^G; G,  \Sigma) = \Sigma( V^G; HE, \Sigma(U^H; D, P) ) =\Sigma(V^G; DE, P).   
\end{equation*}
\par 
Since $DE$ is finite and $P$ is an algebraic subgroup of $U^D= V^{DE}$, 
we conclude that $\Sigma$ is an algebraic group subshift of finite type of $V^G$. 
\end{proof}

\subsection{Application on groups of algebraic Markov type}  

As an immediate application of Theorem~\ref{t:algebraic-group-Z-d} and Proposition~\ref{p:main-by-finite}, we obtain the following property on the class of groups of algebraic Markov type. 
\begin{corollary} 
\label{c:markov-group-extension}
 Let $0 \to H \to G  \to C \to 0$ be an extension of countable groups. 
 Suppose that $C$ is cyclic and $H$ is a group of algebraic Markov type.
 Then $G$ is also a group of algebraic Markov type. 
 \qed  
\end{corollary}

\begin{proof}[Proof of Theorem~\ref{t:main-polycyclic-finite}]
The argument below is standard and similar to the proof of \cite[Theorem~4.2]{schmidt-book}. 
Suppose that $G$ is a countable, polycyclic-by-finite group. 
Then $G$ admits a subnormal series $G=G_n \supset G_{n-1} \supset \cdots \supset G_0=\{1_{G_0}\}$ 
whose factors are cyclic groups. 
Since a trivial group is of algebraic Markov type (see the remark following Definition~\ref{d:markov-type-alg}), 
an immediate induction 
using Theorem~\ref{t:algebraic-group-Z-d} and Proposition~\ref{p:main-by-finite} 
shows that $G_0,   \cdots, G_n=G$ are all of algebraic Markov type. 
The proof of Theorem~\ref{t:main-polycyclic-finite} is completed. 
\end{proof}

\section{Inverse images of homomorphisms of algebraic group subshifts}
\label{s:applications}

For the proof of Theorem~\ref{c:intro-algebraic-group-abelian}, we remark that  
by Theorem~\ref{t:intro-sft-property-alg-grp}, $\Sigma_1$ and $\Sigma_2$  
are algebraic group subshifts of finite type of $U^G$ and   $V^G$ respectively.  
By Theorem~\ref{t:sofic-alg-grp}, 
the algebraic group sofic subshift $\tau(\Sigma_1)$ is in fact an algebraic group subshift of $V^G$. 
Hence it is also  an algebraic group subshift of finite type of $V^G$ by Theorem~\ref{t:intro-sft-property-alg-grp}. 
\par  
The rest of the proof of Theorem~\ref{c:intro-algebraic-group-abelian} is a direct 
consequence of Theorem~\ref{t:intro-sft-property-alg-grp} and the following general result. 

\begin{proposition}
\label{p:ker} 
Let $G$ be a countable group and let $U, V$ 
be algebraic groups over an algebraically closed field. 
Let $\Sigma$ be an algebraic group subshift of $V^G$ and let 
$\tau \in \Hom_{U, V G\text{-algr}}(U^G, V^G)$. 
Then $\tau^{-1}(\Sigma)$ is an algebraic group subshift of $U^G$.  
\end{proposition}

\begin{proof} 
Since $\tau$ is $G$-equivariant and is continuous with respect to the prodiscrete topology, 
$ \tau^{-1}(\Sigma)$ is clearly $G$-invariant and is closed in $U^G$ with respect to the prodiscrete topology. 
\par 
Since $\tau \in \Hom_{U, V G\text{-algr}}(U^G, V^G)$, there exists a finite subset $M \subset G$ containing $1_G$ 
and a homomorphism of algebraic groups $\mu \colon U^M \to V$ such that 
\begin{equation} 
\label{e;M-ker} 
\tau(x)(g) = \mu( (g^{-1}x)\vert_M)  \quad  \text{for all } x \in B^G \text{ and } g \in G. 
\end{equation}
Let $F \subset G$ be finite subset. Then $\Sigma_F$ is an algebraic subgroup of $U^{FM}$ 
and we have a map (cf.~\cite[Section~3]{ccp-goe-2020}): 
\[
\tau_{F}^{+} \colon U^{FM} \to V^F 
 \]
defined by  $\tau_F^{+}(c) = \tau(u)\vert_F$ for every $c \in U^{FM}$ and every $u \in U^G$ such that $u \vert_{FM}=c$.  
Then $\tau_F^{+}$ is a certain fibered product over $g \in F$ of the 
homomorphisms of algebraic groups $U^{gM} \to V^{\{g\}}$ naturally induced by $\mu$. 
Hence, $\tau_{F}^{+}$ is a homomorphism of algebraic groups for every finite subset $F \subset G$. 
\par 
Let $E \subset G$ be a fixed finite subset. 
Since $G$ is countable, there exists an increasing sequence $(E_n)_{n \geq 0}$ 
of finite subsets of $G$ containing $E$ and such that $G= \bigcup_{n \geq 0} E_n$. 
\par 
For every $n \geq 0$, let  $\psi_n \colon U^{E_nM} \to U^E$ be the canonical homomorphism 
of algebraic groups induced by the inclusion $E \subset E_nM$ (cf.~Lemma~\ref{l:restriction-map-also-alg}). 
For every $n \geq 0$, we define respectively algebraic subgroups of $U^{E_nM}$ and $U^E$ by 
\begin{align}
\label{e:w-z-ker}
W_n \coloneqq \left(\tau_{E_n}^{+}\right)^{-1} (\Sigma_{E_n}), 
\quad \quad   
Z_n   \coloneqq \psi_n(W_n).
\end{align} 
\par 
 We claim that $(\tau^{-1}(\Sigma))_E = \bigcap_{n \geq 0} Z_n$. 
Indeed, 
if $z \in (\tau^{-1}(\Sigma))_E$ then there exists $x \in U^G$ such that $z= x \vert_E$ 
and $\tau(x) \in \Sigma$  thus $x\vert_{E_nM} \in W_n$ and $z=\psi_n(x\vert_{E_nM}) \in Z_n$ 
for all $n \geq 0$. Hence, $(\tau^{-1}(\Sigma))_E \subset \bigcap_{n \geq 0} Z_n$. 
\par   
For the converse implication, let $z \in \bigcap_{n \geq 0} Z_n$. 
For every $n \geq 0$, we define $S_n \coloneqq \psi_n^{-1}(z) \cap W_n$. 
Then $S_n$ is a translate of the algebraic subgroup $\Ker(\psi_n) \cap W_n$ of $W_n$. 
Since $z \in Z_n$, the set $S_n$ is nonempty for every $n \geq 0$. 
\par 
We thus obtain an inverse system $(S_n)_{n \geq 0}$ whose transition maps are the restrictions to $S_m$ 
of the canonical homomorphism of algebraic groups $U^{E_mM} \to U^{E_n M}$ 
(cf.~Lemma~\ref{l:restriction-map-also-alg}) for all $m \geq n \geq 0$. 
\par 
By Lemma~\ref{l:inverse-limit-alg-grp}, there exists $y \in \varprojlim_{n \geq 0} S_n \subset U^G$. 
Since $y\vert_{E_nM} \in \psi_n^{-1}(z)$ for all $n \geq 0$, we find that $y\vert_E = z$. 
As $y\vert_{E_nM} \in W_n$ for every $n \geq 0$, we infer from \eqref{e:w-z-ker} 
that $\tau(y)\vert_{E_n} = \tau_{E_nM}^{+}(y\vert_{E_nM}) \in \Sigma_{E_n}$. 
Hence, it follows Lemma~\ref{l:closed-lim} that $\tau(y) \in \varprojlim_{n \geq 0} \Sigma_{E_n} = \Sigma$. 
We deduce that $(\tau^{-1}(\Sigma))_E \supset \bigcap_{n \geq 0} Z_n$. 
This proves the claim that $(\tau^{-1}(\Sigma))_E = \bigcap_{n \geq 0} Z_n$.  
\par 
By the Noetherinaity of the Zariski topology, $\bigcap_{n \geq 0} Z_n$ is an algebraic subgroup 
 of $U^E$ and thus so is $(\tau^{-1}(\Sigma))_E$. 
We can thus conclude that $\tau^{-1}(\Sigma)$ is an algebraic group subshift of $U^G$. 
\end{proof}

\section{Generalizations} 
\label{s:generalization} 
In this section, we obtain axiomatic  generalizations of our main results where the alphabets can now be 
taken for example as Artinian groups or Artinian modules over a ring. 

\subsection{Admissible Artinian group structures}  
\begin{definition} 
\label{d:general-artinian-structure}
Let $\Gamma$ be a group. 
An \emph{admissible Artinian structure} on $\Gamma$ is a sequence $\HH = (\HH_n)_{n \geq 1}$ where every 
$\HH_n$ is a collection of subgroups of $\Gamma^n$ with the following stability properties: 
\begin{enumerate} 
\item 
$ \{1_\Gamma\}, \Gamma  \in \HH_1$; 
\item 
for every $m \geq n \geq 1$ and for every projection $\pi \colon \Gamma^m \to \Gamma^n$ induced by an injection 
$\{1, \cdots, n\} \to \{ 1, \cdots, m\}$, we have $\pi(H_m) \in \HH_n$ and $\pi^{-1}(H_n) \in \HH_m$ 
for every $H_m \in \HH_m$ and $H_n \in \HH_n$;   
\item 
for every $n \geq 1$ and $H, K\in \HH_n$, we have $H \cap K \in \HH_n$; 
\item 
for every $n \geq 1$, every descending sequence $(H_k)_{k \geq 0}$ of subgroups of $\Gamma^n$, where  
 $H_k \in \HH_n$ for every $k \geq 0$, eventually stabilizes. 
\end{enumerate} 
\par 
In this case, we say that  $(\Gamma, \HH)$, or simply $\Gamma$ when there is no possible confusion, 
is an \emph{admissible Artinian group structure}. 
For every $n \geq 1$, elements of $\HH_n$ are called \emph{admissible subgroups} of $\Gamma^n$. 
\end{definition}

We have the following simple observation:  

\begin{lemma}
\label{l:induced-admissible}
Let $\Gamma$ be a group with an admissible Artinian structure $ (\HH_n)_{n \geq 1}$ and 
let $H \subset \Gamma$  be an admissible subgroup. The following holds: 
\begin{enumerate} [\rm (i)]
\item   
for every $m \geq 1$, the group $\Gamma^m$ admits 
an admissible Artinian group structure given by $(\HH_{mn})_{n \geq 1}$; 
\item 
$H$ admits an induced admissible Artinian structure $(\HH'_n)_{n \geq 1}$ given by 
$\HH'_n \coloneqq \{ H^n \cap H_n \colon H_n \in \HH_n \}$.  
\end{enumerate}
 \end{lemma} 

\begin{proof}
The point (i) is straightforward. 
For (ii), let us check the conditions in Definition~\ref{d:general-artinian-structure}. 
The conditions (1) and (3)  are trivial. 
For (2), let $m \geq n \geq 1$ and let $\{1, \cdots, n\} \to \{ 1, \cdots, m\}$ 
be an injection.  
 Let $\pi \colon \Gamma^m \to \Gamma^n$ and let 
 $\pi_H  \colon H^m \to H^n$ be  the induced  projections. 
 Let $n \in \HH_n$ and $H_m \in \HH_m$.  
 Since $H$ is an admissible subgroup of $\Gamma$,  
 so are $H^m \subset \Gamma^m$, $H^n \subset \Gamma^n$ by the condition (2) for $\Gamma$.  
 Hence, $H^m \cap H_m \in \HH_m$ and $H^n \cap H_n \in \HH_n$. 
 By the condition (2) for $\Gamma$, it follows that  $\pi_H (H^m \cap H_m) 
 = \pi(H^m \cap H_m) \in \HH_n$ and similarly 
$\pi_H^{-1}(H^n \cap H_n) = H^m \cap \pi^{-1}( H_n)  \in \HH_m$. 
Thus (2) is verified. 
 \par 
 For (4), let $k \geq 1$ and let $(H^k \cap H_n)_{n \geq 1}$ be a descending sequence 
 where $H_n \in \HH_k$ for every $n \geq 1$. By replacing $H_n$ by $\cap_{1 \leq m \leq n} H_m \in \HH_k$ 
 for every $n \geq 1$, we find that $(H^k \cap H_n)_{n \geq 1}$ is a descending sequence of elements in $\HH_k$.  
  Thus, $(H^k \cap H_n)_{n \geq 1}$ eventually stabilizes by the condition (4) for  $\Gamma$. 
 \end{proof}

\begin{example}
Let $\Gamma$ be a group equipped with a collection $\AAA$ of its subgroups which is stable by taking intersection  
and such that $\{1_\Gamma\}, \Gamma \in \AAA$ and every descending sequence of subgroups taken in $\AAA$ eventually stabilizes. 
Then $\AAA$ induces an admissible Artinian structure  $\HH= (\HH_n)_{n \geq 1}$ on $\Gamma$  
defined by  
$\HH_n   \coloneqq \AAA \times \cdots \times \AAA$ ($n$-times) 
for every $n \geq 1$.  
\par 
Remark that $\HH$ is the smallest admissible Artinian structure of $\Gamma$ such that 
every elements of $\AAA$ is an admissible subgroup of $\Gamma$. 
\end{example}

\begin{definition}
Let the notations be as above. We say that $\HH=  (\HH_n)_{n \geq 1}$ is the \emph{induced product admissible Artinian structure} on $\Gamma$ by $\mathcal{A}$. 
\end{definition}

\begin{example} 
Every algebraic group $V$ over an algebraically closed field, 
resp. every compact Lie group $W$, admits a canonical admissible Artinian structure given by all algebraic subgroups of $V^n$, 
resp. by all closed subgroups of $W^n$,  
for $n \geq 1$.  
\end{example}

\begin{example}
\label{example:canonical-group-admissible} 
Recall that a group $\Gamma$ is \emph{Artinian} if every descending sequence of subgroups of $\Gamma$ eventually stabilizes. 
 \par 
For example,   finite groups are Artinian and for every prime number $p$, 
the  subgroup $\mu_{p^\infty} \coloneqq \{ z \in \C^* \colon \exists\, n \geq 0, \, z^{p^n} = 1  \}$ 
of the multiplicative group $(\C^*, \times)$ is Artinian.   
\par 
 It is well-known that finite direct products of Artinian groups are Artinian (see, e.g., the proof of \cite[Lemma~4.3]{kitchens-schmidt}). 
\par 
Then, it is clear that every Artinian group $\Gamma$ admits a canonical admissible Artinian structure given by all subgroups of $\Gamma^n$ for every $n\geq 1$.  

\end{example}

\begin{example} 
\label{example:canonical-module-admissible} 
Similarly, every Artinian (left or right) 
module $M$ over a ring $R$ is equipped with a canonical admissible Artinian structure given by all $R$-submodules of $M^n$ for all $n \geq 1$. 
\end{example}

\begin{definition} 
\label{d:admissible-homomorphism}
Let $\Gamma$ and $\Gamma'$ be admissible Artinian group structures. 
An admissible Artinian group structure $\mathcal{H}$ 
on $\Gamma \times \Gamma'$ is said to be \emph{compatible} with $\Gamma$ and $\Gamma'$ if 
the following holds: 
\begin{enumerate} [\rm (*)] 
\item 
for every integer $m \geq 1$ and for the canonical projections $p \colon (\Gamma \times \Gamma')^m \to \Gamma^m$ and
$p' \colon (\Gamma \times \Gamma')^m \to (\Gamma')^m$, and 
for all admissible subgroups $P \subset (\Gamma \times \Gamma')^m$ and $Q \subset \Gamma^m$ and $Q'\subset (\Gamma')^m$, 
the groups $p(P)$, $p'(P)$ are respectively  admissible subgroups of $\Gamma^m$ and $(\Gamma')^m$
  and the groups 
 $p^{-1}(Q)$, $(p')^{-1}(Q')$ are admissible subgroups of $(\Gamma \times \Gamma')^m$. 
\end{enumerate} 
\par 
A homomorphism of abstract groups $\varphi \colon \Gamma \to \Gamma'$ 
is said to be \emph{admissible} 
 if there exists a compatible Artinian group structure on $\Gamma \times \Gamma'$
such that 
the graph $\{ (x, \varphi(x)) \colon x \in \Gamma \}$ 
is an admissible subgroup of $\Gamma \times \Gamma'$; 
\end{definition}

Several useful remarks are  in order. 

\begin{remark} 
Homomorphisms of algebraic groups, homomorphisms of compact Lie groups, 
homomorphisms of Artinian groups, and 
morphisms of $R$-modules are all admissible with the canonical admissible Artinian 
structures of algebraic groups, compact Lie groups, Artinian groups, and $R$-modules respectively that are described in the  above examples.  
\end{remark}

\begin{remark}
Given an admissible Artinian group structure $(\Gamma, \HH)$ as in Definition~\ref{d:general-artinian-structure}, 
it is clear that the   projections $\pi \colon \Gamma^m \to \Gamma^n$ induced by injections 
$\{1, \cdots, n\} \to \{ 1, \cdots, m\}$ are admissible homomorphisms of admissible Artinian group structures on $\Gamma^m$ 
and $\Gamma^n$ for every $m \geq n \geq 1$. Moreover, 
by Lemma~\ref{l:induced-admissible}, the restrictions 
 of the canonical projections $\pi \colon \Gamma^m \to \Gamma^n$ to admissible subgroups  
 of $\Gamma^m$ are also admissible homomorphisms. 
\end{remark}

 \subsection{Finiteness properties of admissible group shifts}

 We first extend both the definition of \emph{group shifts} with finite group alphabets in the literature (cf.~\cite{fiorenzi-periodic}) and 
 Definition~\ref{d:alg-subgroup-shift} of algebraic group subshifts   
to the case where the alphabets are admissible Artinian group structures as follows. 

\begin{definition}
\label{d:admissible-subgroup-shift}
Let $G$ be a group. Let $A$ be an admissible Artinian group structure. 
A subshift $\Sigma \subset A^G$ is called an \emph{admissible group subshift}  
 if it is closed in $A^G$ with respect to the prodiscrete topology and  
if the restriction $\Sigma_E \subset A^E$ is an admissible subgroup 
for any finite subset $E \subset G$.        
\end{definition}
\par 
In the above definition and in what follows, the admissible Artinian structure of $A^E$ is induced by that of $A^{\{1, \cdots, |E|\}}$ via an arbitrary bijection 
$\{1, \cdots, |E|\} \to E$.  
\par 
\begin{example} 
\label{example:canonical-admissible-for-intro} 
Let $G$ be a group and let $A$ be an Artinian group (resp. an Artinian module over a ring  $R$). 
Then $A^G$ is naturally an abstract group with pointwise group operations (resp. an $R$-module with pointwise group operations and with pointwise $R$-action).  
Let $\Sigma$ be an abstract subgroup (resp. an $R$-submodule) of $A^G$ which is also a closed subshift. 
Then  
$\Sigma$ is   an admissible subshift of $A^G$ with respect to 
 the canonical admissible Artinian structure on $A$ (cf.~Example~\ref{example:canonical-group-admissible}, Example~\ref{example:canonical-module-admissible}). 
 \par
 Likewise, algebraic group subshifts are also admissible subshifts 
 with respect to the canonical admissible Artinian structure on algebraic groups. 
\end{example}

 As an application of the proofs of our main theorems, we obtain the following general result: 

\begin{theorem} 
\label{t:artinian-general} 
Let $G$ be a polycyclic-by-finite group 
and let $A$ be an admissible Artinian group structure. 
Then every admissible group subshift of $A^G$ 
is a subshift of finite type. Moreover, there exists a finite subset $D \subset G$ and 
an admissible subgroup $W$ of $A^D$ such that $\Sigma = \Sigma(A^G ; D, W)$.  
\end{theorem}
 \par

Before giving the proof of Theorem~\ref{t:artinian-general}, we establish some 
key lemmata on inverse systems 
of admissible Artinian group structures.   

\begin{lemma}
\label{l:lemma-artinian-1} 
Let $\Gamma$ be an admissible Artinian group structure. 
Let $(X_n)_{n \geq 0}$ be a descending sequence of left translates of admissible subgroups of $\Gamma$. 
Then the sequence $(X_n)_{n \geq 0}$ eventually stabilizes. 
\end{lemma}

\begin{proof}
By hypothesis, we can find a sequence of elements $(g_n)_{n \geq 0}$ of $\Gamma$ 
and a sequence of admissible subgroups $(H_n)_{n \geq 0}$ of $\Gamma$ such that 
$X_n = g_n H_n$ for every $n \geq 0$. 
Fix an integer $n \geq 0$, 
it follows that $g_{n+1} H_{n+1} \subset g_n H_n$. 
Then $g_n^{-1} g_{n+1} H_{n+1} \subset H_n$ thus $g_n^{-1} g_{n+1}\in H_n$ and therefore $g_n^{-1} g_{n+1} = h_n \in H_n$ 
for some $h_n \in H_n$. Hence, $g_{n+1}^{-1} g_n =h_n^{-1}$ and we find 
\[
H_{n+1} = g_{n+1}^{-1}g_{n+1} H_{n+1} \subset   g_{n+1}^{-1}g_n H_n = h_n^{-1}  H_n = H_n. 
\]
\par 
Consequently, we obtain a descending sequence $(H_n)_{n \geq 0}$ of admissible subgroups of $\Gamma$. 
Since $\Gamma$ is an admissible Artinian group structure, 
there exists an integer $N \geq 0$ such that $H_n = H_N$ for al $n \geq N$. 
\par 
Finally, remark that if $g,h \in \Gamma$ and $H \subset \Gamma$ is a subgroup such that $g H \subset hH$ then 
necessary $gH= hH$. Applying this remark to the sequence $(g_nH_n)_{n \geq N}$, we can thus conclude that $X_n = X_N$ for all $n \geq N$ and the proof is completed. 
\end{proof}

\begin{lemma}
\label{l:inverse-limit-artinian-general} 
Let $(\Gamma_i, \varphi_{ij})_{i,j \in I}$ 
be an inverse system indexed by a countable directed set $I$, 
where every $\Gamma_i$ is an admissible Artinian group structure and the transition maps  
$\varphi_{i j} \colon \Gamma_j \to \Gamma_i$ are admissible homomorphisms for all $i \prec j$. 
Suppose that $X_i$, for every $i \in I$, is  
a left translate of an admissible subgroup of $\Gamma_i$ and that $\varphi_{ij}(X_j)\subset X_i$ 
for all  $i\prec j$ in $I$. 
Then the induced inverse subsystem $(X_i)_{i \in I}$ satisfies  
\begin{equation*}
\varprojlim_{i \in I} X_i \neq \varnothing. 
\end{equation*} 
\end{lemma}

\begin{proof}
 The  proof below is a straightforward modification of \cite[Proposition~4.2]{phung-2018} using the Mittag-Leffler condition argument. 
 We give the details here for sake of completeness. 
 \par   
First, since $I$ is countable, we can suppose without loss of generality that 
$(I , \prec) = \left( \{ n  \in \Z \colon n \geq 0\}, \leq \right)$  
as  directed sets (see~\cite[Proposition~4.2]{phung-2018}). 
For every integer $n \geq 0$, we define $Y_n \subset \Gamma_n$ by 
\[
Y_n \coloneqq \cap_{m \geq n }\, \varphi_{n m}(X_m).
\] 
\par 
By hypotheses, we deduce  that $(\varphi_{nm}(X_m))_{m \geq n}$ is a descending sequence of left translates of admissible subgroups of $\Gamma_n$. 
Since $\Gamma_n$ is an admissible Artinian group structure for every $n \geq 0$, 
it follows from Lemma~\ref{l:lemma-artinian-1} 
that we can define, by an immediate induction on $n$, an increasing   
sequence of integers $(k_n)_{n \geq 0}$ 
such that 
$k_n \geq n$ and for all $k \geq k_n$, we have 
$Y_n = \varphi_{n k}(X_k) \neq \varnothing$. 
\par 
For all integers $m \geq n \geq 0$,  consider the induced maps   
\[
f_{n m} \coloneqq \varphi_{n m} \vert_{Y_m}\colon Y_m \to Y_n 
\]
and remark form the choice above of the integers $k_m \geq k_n$  that  
\begin{align*}
f_{n m}(Y_m)
& = \varphi_{nm} \left(  \varphi_{mk_m}(X_{k_m}) \right)
 =  \varphi_{nk_m}(X_{k_m})
 =  Y_n. 
\end{align*}
\par 
We thus obtain the universal inverse system $(Y_n, f_{n m})$ associated with $(X_n, \varphi_{nm})$ in which   
the transition maps are surjective maps between nonempty spaces. 
Since the directed set is countable, 
it follows that $\varprojlim_{n \geq 0} Y_n$ is nonempty. Therefore, we conclude that 
$\varprojlim_{n \geq 0} X_n = \varprojlim_{n \geq 0} Y_n$ is nonempty. 
The proof is completed.   
\end{proof}

\begin{proof}[Proof of Theorem~\ref{t:artinian-general}] 
The proof is exactly the same, \emph{mutatis mutandis},  
as the proof of Theorem~\ref{t:main-polycyclic-finite}. 
It suffices to replace  {algebraic groups},  {algebraic subgroups}, and  {homomorphisms of algebraic groups} respectively by 
 {admissible Artinian group structures},  {admissible subgroups}, and  
 {admissible homomorphisms} of admissible Artinian group structures. 
Applications of Lemma~\ref{l:inverse-limit-alg-grp} in the proof should also be replaced everywhere 
by applying Lemma~\ref{l:inverse-limit-artinian-general} instead. 
Remark also that the Noetherianity of the Zariski topology on algebraic groups is now replaced 
by the descending chain condition in the definition of admissible Artinian group structures.  
\end{proof}

\subsection{Generalizations of other main results}

In fact, by repeating the exact same definitions and proofs in the present paper 
to admissible Artinian  group structures instead 
of algebraic groups, as we did for 
Definition \ref{d:admissible-subgroup-shift} and  Theorem~\ref{t:artinian-general}, 
all the results of the paper remain valid. 
We omit the details of similar straightforward verifications of the following principal results.

\begin{theorem} 
\label{t:artinian-general-sft} 
Let $G$ be a countable group 
and let $A$ be an admissible Artinian group structure. 
Let $D \subset G$ be a finite subset and let $P \subset A^D$ be an admissible subgroup. 
Then the subshift of finite type $\Sigma(A^G; D, P)$ is an admissible group subshift of $A^G$.  
\end{theorem}

\begin{proof}
Similar to the proof of \cite[Corollary~6.3]{cscp-2020}. 
\end{proof}
\par 
Let the notations and hypotheses be as in Theorem~\ref{t:artinian-general-sft}, such a subshift 
$\Sigma(A^G; D, P)$ is called an \emph{admissible group subshift of finite type} of $A^G$.  
\par 
\begin{proposition} 
\label{p:markov-alg-admissible} 
Let $G$ be a countable group and let $A$ be an admissible Artinian group structure. 
The following are equivalent: 
\begin{enumerate} [\rm (a)] 
\item 
every admissible group subshift of $A^G$ is an admissible group subshift of finite type;  
\item 
every descending sequence of admissible group subshifts of $A^G$ eventually stabilizes.
\end{enumerate} 
 \end{proposition}
 
 \begin{proof}
 Similar to the proof of Proposition~\ref{c:markov-alg}.  
 \end{proof}

\begin{corollary}
\label{c:artinian-general} 
Let $G$ be a polycyclic-by-finite group 
and let $A$ be an admissible Artinian group structure. 
Then every descending sequence of admissible group subshifts of $A^G$ eventually stabilizes.  
\end{corollary}

\begin{proof}
It is a consequence of Proposition~\ref{p:markov-alg-admissible}  
and Theorem~\ref{t:artinian-general}.  
\end{proof}
\par 
Now let $G$ be a group and let $A, B$ be admissible Artinian group structures.  
\begin{definition}
A cellular automaton $\tau \colon A^G \to B^G$ is an \emph{admissible group cellular automaton}  
if $\tau$ admits a memory set $M$ whose associated local defining map 
$\mu \colon A^M \to B$ is an admissible homomorphism of admissible Artinian group structures (cf.~Definition~\ref{d:admissible-homomorphism}). 
\end{definition} 
\par 
In this case, consider an arbitrary finite subset $F \subset G$ and the induced map 
(cf.~\cite[Section~3]{ccp-goe-2020}) 
\[
\tau_{F}^{+} \colon A^{FM} \to B^F
\]
given by 
$\tau_F^{+}(c) = \tau(x) \vert_F$ for every $c \in A^{FM}$ and $x \in A^G$ such that $x \vert_{FM}=c$. 
Then we have the following: 
\begin{lemma} 
$\tau_{F}^{+}$ 
is an admissible homomorphism of Artinian group structures. 
\end{lemma} 

\begin{proof} 
Indeed, for every 
$c \in A^{FM}$, we have 
$\tau_{F}^{+} (c) (g) = \mu(g^{-1}c  \vert_M)$ for 
every $g \in F$ where we recall that  $g^{-1}c \in A^{g^{-1}FM}$ is defined by  
$(g^{-1}c)(h) \coloneqq c(gh)$ for every $h \in g^{-1}FM$. 
For each $g \in F$, let $\Gamma(g) \subset A^{gM} \times B^{\{ g \} }$ be the graph 
of the map $\mu_g \colon A^{gM} \to B^{ \{ g\} }$ defined by $\mu_g (x) \coloneqq \mu(g^{-1}x)$ and let $\pi_g \colon A^{FM} \times B^F \to A^{gM} \times B^{\{ g \} }$ 
be the canonical projection. 
\par 
By Definition  \ref{d:admissible-homomorphism}, there is an Artinian group structure $\HH$ of $A^M \times B$ compatible with the Artinian group structures of $A^M$ and $B$ so that the graph $\Gamma(g)$ is an admissible subgroup of $A^{gM} \times B^{\{ g \}}$ for every $g \in F$. 
\par 
Let $\iota \colon FM \to F \times M$ be any set-theoretic injection.  Let $\pi \colon A^{F \times M} \times B^F \to A^{F \times M}$ be the canonical projection. 
Then we can write $A^{FM} \times B^F = \pi^{-1} (A^{FM}\times \{1_A\}^{F\times M \setminus \iota(FM)})$ which is an admissible subgroup of $A^{F \times M} \times B^F$ 
by Property $(*)$ in  Definition~\ref{d:admissible-homomorphism}.
Hence,  $\HH$ induces a compatible Artinian group structure on $A^{FM} \times B^F$  (cf.~Lemma~\ref{l:induced-admissible}.(ii)). 
\par 
Then it is not hard to see that 
the graph of $\tau_{F}^{+}$ 
is the following intersection:  
\[
\Gamma_{\tau_{F}^{+}} = \bigcap_{g \in F}  
\pi_g^{-1}(\Gamma(g)). 
\] 
\par 
On the other hand, it follows  from Definition~\ref{d:admissible-homomorphism} that for every $g \in F$,   $\pi_g^{-1}(\Gamma(g))$ is an admissible subgroup of $A^{FM} \times B^F$  both regarded as admissible subgroups of $A^{F\times M} \times B^F$ . As intersections of admissible subgroups are also admissible subgroups, we deduce that 
$\Gamma_{\tau_{F}^{+}}$ is an admissible subgroup of $A^{FM} \times B^F$. Hence, 
 $\tau_{F}^{+} $ is an admissible homomorphism. 
\end{proof} 

\begin{theorem} 
\label{c:algebraic-group-admissible} 
Let $G$ be a polycyclic-by-finite group. 
Let $A, B$ be admissible Artinian group  structures. 
Let $\Sigma_1$ and  $\Sigma_2$ be  respectively admissible group subshifts 
of $A^G$ and $B^G$.  
Let $\tau \colon A^G \to B^G$ be an  admissible group cellular automaton.   
Then $\tau^{-1}(\Sigma_2)$ and  $\tau(\Sigma_1)$  
are respectively admissible group subshifts  of finite type of $A^G$ and $B^G$.  
\end{theorem} 

 \begin{proof} 
 Similar to Theorem~\ref{c:intro-algebraic-group-abelian}.  
 \end{proof}
 
\begin{corollary}
\label{c:limit-set-alg-admissible} 
Let $G$ be a polycyclic-by-finite group and let $A$ be an  admissible Artinian group structure. 
Let $\Sigma$ be an admissible group subshifts of $A^G$. 
Let  $\psi \colon A^G \to A^G$ be an admissible group cellular automaton. 
Suppose that $\psi(\Sigma) \subset \Sigma$ and let $\tau \colon \Sigma \to \Sigma$ be the restriction map.   
Then the limit set $\Omega(\tau)$ is an  admissible group subshift of finite type of $A^G$ and $\tau$ is stable.  
\end{corollary}

\begin{proof}  
Similar to the proof of  Corollary~\ref{c:limit-set-alg-grp}.  
\end{proof}

\section{Density of periodic configurations in algebraic subshifts and admissible group subshifts} 
\label{s:density} 
In this section, we obtain  
some generalizations to the algebraic setting of a known density result in the classical setting with finite alphabets. 
\par 
We then give some direct consequences of our main results 
on the density of periodic configurations in algebraic group subshifts and 
more generally in admissible group subshifts. 
A counter-example is given at the end of the section. See also  \cite{phung-2020-shadow} for another  application of our main results on the  pseudo-orbit tracing property. 

\subsection{Languages, words and irreducibility of subshifts}  
Let $A$ be a finite set.
We recall some useful terminologies concerning languages and words of a given subshift $\Sigma \subset A^\Z$. 
Let $A^*$ be the free monoid generated by $A$. 
Elements of $A^*$ are  finite words $w = a_1 a_2 \cdots a_k$, where $n \geq 0$, with letters $a_i \in A$ for all $1 \leq i \leq k$, 
The concatenation of words defines the monoid structure on $A^*$ whose identity 
element is the empty word denoted by $\varepsilon \in A^*$. 
For every word  $w = a_1 a_2 \cdots a_k$,  we define its \emph{length} by $\vert w  \vert \coloneqq k$. 
\par 
Given a word $w = a_1 a_2 \cdots a_n \in A^*$ of length $n\geq 1$, we define a periodic configuration $w^\infty \in A^\Z$ given by 
$w^\infty(i + k n) = a_i$ for all  $k \in \Z$ and $1 \leq i \leq n$.  
\par
A word $w \in A^*$ is a \emph{subword} of a configuration $x \in A^\Z$ if either $w$ is the empty word or there exist integers $m \geq$ such that
$w = x(n)x(n+1) \cdots x(m)$.
\par
Consider now a subshift $X \subset A^\Z$.
The \emph{language} of $X$ is the subset $L(X) \subset A^*$  consisting of all words $w \in A^*$
such that $w$ is a subword of some configuration in $X$.
\par 
Now let $G$ be a group and let $A$ be a set. 
Let $\Sigma \subset A^G$ be a subshift. 
Then $\Sigma$ is \emph{irreducible} if it  for any $y, z \in \Sigma$ and 
any finite subsets $E,F \subset G$, there exist $x \in \Sigma$ and $g \in G$ 
such that $x\vert_E = y\vert_E$ and $(gx)\vert_F = z\vert_F$. 
If $G=\Z$, this means that for all $u,v \in L(\Sigma)$,  
there exists $w \in L(\Sigma)$ such that $uwv \in L(\Sigma)$. 
\par 
We say that $\Sigma$ is \emph{strongly irreducible} 
if there exists a finite subset $\Delta \subset G$ with the following property.  
If $E, F$ are finite subsets of $G$ such that $E \cap F\Delta = \varnothing$, 
 then, for any $y, z \in \Sigma$, 
 there exists $x \in \Sigma$ such that 
$x\vert_E = y\vert_E$ and $x\vert_F = z\vert_F$. 
 
\subsection{Density of periodic  configurations in algebraic $W$-subshifts} 
As a generalization of strongly irreducible subshifts and of irreducible subshifts of finite type, 
the following definition is introduced in \cite{csc-2011-density} after Benjy Weiss 
(cf.~\cite[Proposition~4.4,~Proposition~4.5]{csc-2011-density}). 
\begin{definition} 
[cf.~\cite{csc-2011-density}]  
\label{def:weiss-subshift} 
Let $A$ be a set. An irreducible subshift $\Sigma \subset A^Z$ is called a  \emph{W-subshift} if there exists an  integer 
 $N \geq 0$  such that for every $u \in  L(\Sigma)$, there exists $c \in L(\Sigma)$ such that $\vert c \vert \leq N$ and 
  $ucu \in L(\Sigma)$. 
\end{definition} 

 We obtain below a generalization to the algebraic setting of a criterion of the density of periodic configurations 
with finite alphabet  (cf.~\cite[Theorem~5.1]{csc-2011-density}). 

\begin{theorem}
\label{density-complete-z}
Let $V$ be a complete algebraic variety  over an algebraically closed field.  
Let  $\Sigma \subset V^\Z$ be an algebraic  $W$-subshift. 
Then $\Sigma$ contains a dense set of periodic configurations.
\end{theorem}

\begin{proof} 
(Benjy Weiss + $\varepsilon$) 
We only need to prove that every word $w  \in L(\Sigma)$ is a subword of some  periodic configuration of $\Sigma$.
\par 
As $\Sigma$ is a $W$-subshift, 
we can find an integer $N \geq 0$  satisfying the condition described in  Definition~\ref{def:weiss-subshift}. 
Let $Z \coloneqq \coprod_{k=1}^{N} V^{\{1, \cdots, k\}}$. 
Since $V$ is a complete algebraic variety by hypothesis, so is $Z$.  
For each word $u\in L(\Sigma)$, consider  the subset $F(u) \subset Z$ which  represents all nonempty words $c \in A^*$ of length at most $N$ 
such that  $u c u \in L(\Sigma)$. 
\par 
By the choice of $N$, the set $F(u)$ is nonempty for every word $u \in L(\Sigma)$. 
We claim that $F(u)$ is in fact an algebraic subvariety of $Z$. To see this, 
let $m= |u|$ and  define for every $1 \leq k \leq N$ a subset of $ V^{\{-m+1, \cdots, k + m\}}$: 
\begin{multline*} 
 uV^{\{1, \cdots, k\}} u \coloneqq 
 \{ x \in V^{\{-m+1, \cdots, k + m\}} \colon  x(-m+1)\cdots x(0) = \\ 
 =x(k+1) \cdots x( k+m) = u \in L(\Sigma) \}. 
\end{multline*} 
\par 
Let $Y_u \coloneqq \Sigma_{\{-m+1, \cdots, k + m\}} \cap uV^{\{1, \cdots, k\}} u \subset V^{\{-m+1, \cdots, k + m\}}$. 
We consider also $\pi_k \colon V^{\{-m+1, \cdots,  k + m\}} \to V^{\{0, \cdots, k\}}$
the canonical projection induced by the inclusion $\{0, \cdots, k\} \subset \{-m+1, \cdots, k+ m\}$.  
Then we have for every $1 \leq k \leq N$ that 
\[
F(u) \cap V^{\{1, \cdots, k\}} = \pi_k (Y_u). 
\] 
\par 
Remark that $\Sigma_{\{-m+1, \cdots, k + m\}}$ is clearly a complete algebraic subvariety of 
$V^{\{-m+1, \cdots, k + m\}}$.  
 On the other hand, $uV^{\{1, \cdots, k\}} u$ is a complete algebraic subvariety  of 
 $V^{\{-m+1, \cdots, k + m\}}$. 
 It follows that $Y_u$ is also 
 is a complete algebraic subvariety  of  
 $V^{\{-m+1, \cdots, k + m\}}$. 
Note that $\pi_k$ is clearly an algebraic morphism. 
Therefore, $F(u)$ is a finite union of algebraic subvarieties of $Z$ so that it is indeed an algebraic subvariety of $Z$.   
\par 
Hence, by the Noetherianity of the Zariski topology, 
there exists a word $u_0 \in L(\Sigma)$ such that there does not exists $u \in L(\Sigma)$ such that 
$F(u) \subsetneqq F(u_0)$. Let $c_0 \in F(u_0)$ be an arbitrary fixed word.  
\par
Suppose that  $v \in A^*$ is such that  $u_0vu_0\in L(\Sigma)$. 
Then it is immediate that $F(u_0vu_0 ) \subset F(u_0)$ and hence  by the choice of $u_0$, we must have
\begin{equation*}
\label{e:fillers}
F(u_0 v u_0) = F(u_0). 
\end{equation*}
In particular, $c_0 \in F(u_0 v u_0)$ and thus $u_0 v u_0 c_0 u_0 v u_0 \in L(\Sigma)$. 
From this point, the rest of the proof is exactly the same as in the proof of \cite[Theorem~5.1]{csc-2011-density}. 
 \end{proof}

\subsection{Some consequences on admissible group subshifts} 

\begin{corollary}
\label{density-alg-grp-z}
Let $A$ be an admissible Artinian group structure. 
Let  $\Sigma \subset A^\Z$ be an irreducible admissible group subshift. 
Then $\Sigma$ contains a dense set of periodic configurations. 
\end{corollary}

\begin{proof} 
Theorem~\ref{t:artinian-general} implies that $\Sigma$ is an admissible group subshift of finite type of $A^\Z$. 
Hence, there exists an interval $M= \{ -k , \dots, k \} \subset \Z$ where $k \geq 0$ such that 
$\Sigma = \Sigma (A^\Z; M, \Sigma_M)$. The argument below is standard 
(see for example \cite{fiorenzi-periodic}). 
\par 
Let $y \in \Sigma$ be a configuration.  
Let $\varepsilon \in A$ be the neutral element of the group $A$ and let $e \coloneqq \varepsilon^M \in A^M$. 
We regard $e \in L(\Sigma)$ as a word of length $|M|$. 
For every $n \geq 0$, define $w_n \coloneqq y\vert_{\{-n, \cdots, n\}} \in L(\Sigma)$.  
Since $\Sigma$ is irreducible, there exists a word $u_n \in L(\Sigma)$ such 
that $eu_nw_n \in L(\Sigma)$. Again, by the irreducibility of $\Sigma$, 
there exists $v_n \in L(\Sigma)$  such that $eu_nw_n v_ne \in L(\Sigma)$. 
Consider the periodic configuration $x_n \in A^\Z$ 
which is a suitable translate of $(eu_nw_n v_n)^{\infty}$ such that  
 $(x_n)\vert_{\{-n, \cdots, n\}} =  w_n= y\vert_{\{-n, \cdots, n\}} \in A^{\{-n, \cdots, n\}}$.  
It is clear that  $(gx_n)\vert_M \in \Sigma_M$ for all $g \in \Z$. 
Thus, $x_n \in \Sigma(A^\Z; M, \Sigma_M)= \Sigma$. 
\par 
Therefore, we obtain a sequence of periodic configurations $(x_n)_{n \geq 0}$ of $\Sigma$ which converges to $y$. 
The proof is thus completed. 
\end{proof} 

The following theorem is proved in \cite[Theorem~1.1]{csc-2011-density} where it was stated for a finite alphabet but 
the proof actually holds for an arbitrary alphabet. 

 \begin{theorem} 
 [cf.~\cite{csc-2011-density}]
\label{t:density-periodic} 
Let $G$ be a residually finite group and let $A$ be a set.
Suppose that $\Sigma \subset A^G$ is a strongly irreducible subshift of finite type
and that there exists a periodic configuration in $X$.
Then $\Sigma$ contains a dense set of periodic configurations.
\qed
\end{theorem}

Since every polycyclic-by-finite group is  residually finite, we obtain the following immediate consequence of Theorem~\ref{t:density-periodic} and   
Theorem~\ref{t:artinian-general}: 

\begin{corollary}
\label{density-alg-grp-res-finite} 
Let $G$ be a polycyclic-by-finite group and let $A$ be an admissible Artinian group structure.   
Let  $\Sigma \subset A^G$ be a strongly irreducible admissible group subshift. 
Then $\Sigma$ contains a dense set of periodic configurations. 
\qed
\end{corollary}

\subsection{A counter-example} 

Let $G$ be a finitely generated abelian group and let $A$ be a finite group. 
By \cite[Corollary~7.4]{kitchens-schmidt}, we know that for every closed subshift $\Sigma \subset A^G$ 
which is also an abstract subgroup, the set of periodic configurations of $\Sigma$ is dense in $\Sigma$ with respect to 
the prodiscrete topology.   
\par 
However, when the alphabet is not a compact space, the set of periodic configurations of an admissible group subshift  
may fail to be dense. 
\begin{example} 
Let $G= \Z$ and let $A= \Q$ be the field of rational numbers. 
Let $D = \{0,1\} \subset G$ and let $P = \{(x_0, x_1) \in A^D \colon x_1 = a x_0\}$ 
for some constant $a \in \Q $ such that $a >1$.  
Consider the subshift of finite type $\Sigma \coloneqq \Sigma(A^G; D, P)$ of $A^G$. 
Then it is clear that $\Sigma = \Q  c$ where $c \in A^G$ is the configuration 
defined by $c(n)= a^n$ for every $n \in \Z$. Hence, $\Sigma$ is also a $\Q$-vector subspace of $A^G$. 
However, the only periodic configuration of $\Sigma$ is the zero-configuration $0^G$ since 
$a^n \to + \infty$ when $n \to + \infty$. 
It follows that $\Sigma$ does not contain a dense subset of periodic configurations. 
\end{example}

\bibliographystyle{siam}
\bibliography{lsalgsft}

\end{document}